\documentclass[11pt]{article}
\usepackage{amsmath}
\usepackage{amssymb}
\usepackage{amsxtra}
\usepackage{amsthm}
\usepackage{mathrsfs}
\usepackage{mathtools}
\usepackage{graphicx}
\usepackage{epstopdf}
\usepackage{subfigure}
\usepackage{tabularx}
\usepackage{array}
\usepackage{float}
\usepackage{xcolor,graphicx,float}
\usepackage{caption}
\usepackage{authblk}
\usepackage[colorlinks,linkcolor=red,anchorcolor=blue,citecolor=blue]{hyperref}
\def\EquationsBySection{\def\theequation
	{\thesection.\arabic{equation}}
	\@addtoreset{equation}{section}}

\usepackage{indentfirst}

\theoremstyle{definition}

\newtheorem{Theorem}{\indent  Theorem}[section]  

\newtheorem{Lemma}[Theorem]{\indent  Lemma}
\newtheorem{Corollary}[Theorem]{\indent  Corollary}
\newtheorem{Definition}{\indent  Definition}

\newtheorem{Remark}[Theorem]{\indent  Remark}

\numberwithin{equation}{section}

\newcommand{\be}{\begin{equation}}
	\newcommand{\ee}{\end{equation}}

\newcommand{\bes}{\begin{equation*}}
	\newcommand{\ees}{\end{equation*}}

\renewcommand{\theequation}{\thesection.\arabic{equation}}

\newcommand{\bean}{\begin{eqnarray*}}
	\newcommand{\eean}{\end{eqnarray*}}

\newcommand\old[1]{}
\makeatother
\usepackage{appendix}
\usepackage{geometry}
\allowdisplaybreaks[3]
\EquationsBySection

\newenvironment{Proof}{\par\noindent {\it Proof of} \enspace\ignorespaces}{\hfill$\square$}

\begin{document}
	\date{}
	\title{Delay rough evolution equations}
\author{

	{\bf   Shiduo Qu$^{}$\thanks{101300285@seu.edu.cn}$\;$   and  Hongjun Gao$^{}$}\thanks{The corresponding author. hjgao@seu.edu.cn}\\
	\footnotesize{School of Mathematics, Southeast University, Nanjing 211189, China  }  
	
}
	\maketitle
	\noindent {\bf \small Abstract}{\small 
		\quad 	In this paper, we accomplish the existence and stability of the solution of a class of delay rough partial differential equations (DRPDEs).   Moreover,  we prove that the solution of DRPDEs can converge to that of RPDEs in sense of some distance as the delay tends to zero.  As applications, we  employ the main results to the study of a class of  delay  stochastic partial differential equations driven by fractional Brownian motion with Hurst parameter $\alpha\in(\frac{1}{3},\frac{1}{2}]$.}



		\section{Introduction}
	
	In this paper, we consider the following delay rough PDEs:
	\begin{align}
		\textup{d}y_{t}&=[A y_{t}+F(y_{t}, y_{t-r})]\textup{d}t+G(y_{t},y_{t-r})\cdot\textup{d}\bar{\mathbf{X}}_{t},\label{eq311}\\
		y_{t}&=\phi_{t}, ~~ t\in[-r,0],\label{eq312}
	\end{align}
	where $A$ generates an analytic $C_{0}$-semigroup $\{S_{t}\}_{t\in \mathbb{R}^{+}}$ on a separable Banach
	space $\mathcal{B}$, $F$  is Lipschitz function, $G$ is three-differentiable function, $r$ denotes the time delay, $\bar{\mathbf{X}}=(X,\mathbb{X},\mathbb{X}(-r))$ is delayed $\alpha$-H\"{o}lder path, $\phi_{t}$ is a controlled path representing the initial value of $y_{t}$. 
	
	The equations (\ref{eq311})-(\ref{eq312})  have two elements that deserve our attention: the delay $r$ and delayed $\alpha$-H\"{o}lder path $\bar{\mathbf{X}}$. On the one hand, the application of delay differential equations is a crucial aspect in various fields of physics, engineering and biology. These equations account for the influence of past states on the current behavior of a system, making them particularly useful in modeling systems that exhibit time lags or delays.  For instance, typical examples can be founded in  climate models \cite{MR2343865}, financial models \cite{MR2120285}, population models \cite{MR2523298}. For more details, please refer to  \cite{MR1386683, MR1415838}. 
	
	On the other hand,  the introduction of $\bar{\mathbf{X}}$ aims to facilitate the study of delay differential equations driven by low regularity term, in particular delay stochastic differential equations driven by Brownian motion with Hurst parameter $\alpha\in(\frac{1}{3},\frac{1}{2})$.  An effective approach to studying these equations is based on the framework of rough paths.    The original rough paths theory was pioneered by  Lyons and Qian \cite{MR1654527,MR2036784}.  Early research efforts in this theory often focused on stochastic differential equations (SDEs) \cite{MR4174393}. In \cite{MR2599193}, Gubinelli and Tindel   broadened the scope of application of rough paths, from SDE to SPDEs. After that, more attention  was paid to the study of rough PDEs. For    RPDEs  driven by nonlinear multiplicative noise,  the mild solution was investigated  based on  the framework of semigroup of linear operator in \cite{MR4040992,MR4299812,MR4097587,MR4431448}. Variation approach and energy method  were employ to derive  the well-posedness  of  RPDEs  with unbounded rough drivers (for instance, transport noise) in  \cite{MR3957994,MR3797622,MR4074697}. Random dynamical system for R(P)DEs is also a topic of significant concern within this field including attractor \cite{MR4385780,MR4608383}, invariant manifold \cite{MR4567460,MR4560602,MR4284415},  Wong--Zakai approximations \cite{MR4551604,MR4266219}. In addition,  numerical schemes  for R(P)DEs  were provided  in \cite{MR4100851,MR2891223,MR4656793}.
	
	Recently, delay SDEs have been especially attractive to scholars. If the equation is driven  by fractional Brownian motion with Hurst parameter $\alpha\in[\frac{1}{2},1)$,      It\"{o} integral, Skorohod  integral  or Young integral is sufficient to treat  the corresponding problems \cite{MR2888697,MR3688537,MR2803093,MR3146369,MR2202322,MR2737158} .  If the Hurst parameter is less than $\frac{1}{2}$,  a treatment is rough path theory.   In \cite{MR2453555}, the authors  
	introduced delayed
	controlled path to present  the definition of delay rough integral based  on  rough path theory. And, they further established an
	existence and uniqueness result for delay differential equations (finite dimensions) driven by a $\alpha$-H\"{o}lder function with $\alpha\in(\frac{1}{3}, 1)$. As a application, their results can be used to study  delay SDEs 
	driven by fractional Brownian motion with Hurst parameter $\alpha \in (\frac{1}{3},1)$. In \cite{MR3236092},  these  results were extended  to the case with $\alpha\in(\frac{1}{4}, 1)$.   In \cite{MR4385646}, the authors showed that  delay rough differential equations can generate random dynamical systems  and  proved  a multiplicative
	ergodic theorem on measurable fields of Banach spaces. In  \cite{MR4251893}, the authors establish the existence of invariant manifolds for random rough delay equations. The framework of multiplicative functional is also an approach to address  similar equations. For more details, please refer to \cite{MR4117091,MR3225810,MR2836524,MR2471936}.
	
	It is worth noting that few results are considered for delay infinite-dimensional differential equations  driven by a $\alpha$-H\"{o}lder function with $\alpha< \frac{1}{2}$. The difficulties  arise in the interaction between space-time regularity  and delay. In \cite{MR2453555}, authors considered the delay ODEs. However, since  delay PDEs involve space variables,
	   how to define the rough integral  is a question worthy of serious consideration.  
	Compared to some result about RPDEs \cite{MR4040992,MR4299812,MR4097587,MR4431448}, 
	the initial value  is a controlled path instead of a point and the nonlinear diffusion is coupled with delay terms, so the techniques used for RPDEs without delay may does not work.
    Hence, it requires overcoming obstacles in  establishing the existence of the solution of RPDEs with delay.  Another difficult is to show the solution of  RPDEs is stability with respect to the delay. Namely, how to prove that the solution of delay RPDEs can converge to that of RPDEs as the delay approaches to zero.  The similar argument for delay differential equations (finite dimensions) was studied in \cite{MR4117091} relying on  the framework of multiplicative functional. To best our knowledge, it is unknown  for delay differential equations (infinite dimensions) based on rough path theory.  Therefore, this problem  is also extremely challenging.
	
	The purpose of this article is overcome such difficulties to  study  (\ref{eq311})-(\ref{eq312}). More precisely,   
	the definition of delayed rough convolution is shown  by employing sewing lemma firstly. Secondly, we prove that there exists a global mild solution for  (\ref{eq311})-(\ref{eq312}) by using a prior estimate and iteration. Thirdly, we further discuss the stability of the solution with respect to the initial value and the driver path.  Finally, we provide a framework to investigate the convergence  of the solution with respect to  the delay.  As a result, we show that  the solution of delay RPDEs can approach to that of RPDEs as the delay tends to zero.     In conclusion, we  generalize previous results from equations without delay to equations with delay, from finite-dimensional differential equations to  infinite-dimensional differential equations.
	
	This article is organized as follows. In Section 2, we introduce useful notations and review  relevant concepts. Section 3 gives some basic lemmas. In Section 4, we show the existence of stability of the solution of RPDEs. Section 5  is devoted to the convergence  of the solution with respect to  the delay. In the final section, we present   examples to illustrate our main results. Some  concepts and proofs are givens in Appendix.

	\section{Preliminaries}
	
	In this paper,
	for simplicity of notations, we write $f\lesssim_{k} g$ if $f\leq C_{k} g$ in which $C_{k}$ is a positive constant dependent of  $k$, and $f\lesssim g$ if $f\leq C g$ in which constant $C$ is independent of any objects that interests us.

	Referring  to the framework of \cite{MR3957994,MR4040992}, we introduce some notations as follows
	\begin{Definition}
		Let $\{\mathcal{B}_{\theta}\}_{\theta\in\mathbb{R}} $ be a family of separable Banach spaces  with norms $\|\cdot\| _{\theta}$.  We call $\{\mathcal{B}_{\theta}\}_{\theta\in\mathbb{R}} $  a 
		monotone family of interpolation spaces 
		if  $\mathcal{B}_{\theta_{2}}$ is densely and continuously embedded into $\mathcal{B}_{\theta_{1}}$  
		for $\theta_{1}\leq\theta_{2}$, and  the following interpolation inequality holds: 
		\begin{align}\label{eq053}
			\|x\|^{\theta_{3}-\theta_{1}}_{\theta_{2}}\lesssim\|x\|^{\theta_{3}-\theta_{2}}_{\theta_{1}}\|x\|^{\theta_{2}-\theta_{1}}_{\theta_{3}},~\text{for}~\theta_{1}\leq\theta_{2}\leq\theta_{3},~x\in\mathcal{B}_{\theta_{3}}.
		\end{align}
	\end{Definition}
	
	Throughout this paper, we 	let $\{\mathcal{B}_{\theta}\}_{\theta\in\mathbb{R}} $ be a monotone family of interpolation spaces, and allow semigroup $\{S_{t}\}_{t\in\mathbb{R}^{+}}$ and spaces $\{\mathcal{B}_{\theta}\}_{\theta\in\mathbb{R}} $ enjoy such property:
	\begin{align}\label{eq050}
		\|S_{t}x\|_{\theta+\sigma}\lesssim t^{-\sigma}\|x\|_{\theta},~~\text{and}~~
		\|S_{t}x-x\|_{\theta}\lesssim t^{\sigma}\|x\|_{\theta+\sigma}, 
	\end{align}
	where $x\in\mathcal{B}_{\theta+\sigma}$, $\theta\in\mathbb{R}$, $\sigma\in[0,1]$. 
	
	For $\sigma_{1},\sigma_{2}\in\mathbb{R}$, $k,m,n\in\mathbb{Z}^{+}_{0}$\footnote{$\mathbb{Z}^{+}_{0}$ denotes all the non-negative integers}, define $\mathcal{C}^{k}_{\sigma_{1},\sigma_{2}}(\mathcal{B}^{m}, \mathcal{B}^{n})$ as the space of functions $F:\mathcal{B}^{m}_{\theta}\rightarrow\mathcal{B}^{n}_{\theta+\sigma_{2}}$ for every $\theta\geq\sigma_{1} $ which are $k$-differentiable with respect to each argument, $F$ is bounded operator and $D^{i}F$ is bounded operator for $i=1,\cdots,k.$

	And, define $\text{Lip}_{\sigma_{1},\sigma_{2}}(\mathcal{B}^{m}, \mathcal{B}^{n})$ as the space of functions $F:\mathcal{B}^{m}_{\theta}\rightarrow\mathcal{B}^{n}_{\theta+\sigma_{2}}$ for every $\theta\geq\sigma_{1} $  which are  Lipschitz continuous with respect to each argument. 

	For  a compact interval $I\in\mathbb{R}$ and $n\in\mathbb{N}$, 
	set $\Delta_{n,I}:=\{(t_{1}, t_{2},\cdots,t_{n}) \in I^{n}:t_{1}\geq t_{2}\geq \dots\geq t_{n}\}$.
	For  a Banach space $U$ with norm $\|\cdot\|_{U}$,  we let $\mathcal{C}_{n}(I,U)$ be the  space of continuous functions from $\Delta_{n,I}$ to $U$, and 
	set 
	\begin{align*}
		\|f\|_{\alpha,U}:&=\sup_{\substack{t_{1},t_{2}\in \Delta_{2,I}\\ t_{1}\neq t_{2}   }  }\frac{\|f_{t_{1},  t_{2}}\|_{U}}{|t_{2}-t_{1}|^{\alpha}},~f\in \mathcal{C}_{2}(I,U),~\alpha>0,\\
		\|f\|_{\alpha_{1},\alpha_{2},U}:&=\sup_{		
			\substack{
			t_{1},t_{2},t_{3}\in\Delta_{3,I}\\  
	  t_{1}\neq t_{2}, t_{2}\neq t_{3}}
}\frac{\|f_{t_{1},t_{2},t_{3}}\|_{U}}{(t_{3}-t_{2})^{\alpha_{1}}  
			(t_{2}-t_{1})^{\alpha_{2}}},~f\in \mathcal{C}_{3}(I,U)~\alpha_{1},\alpha_{2}>0.
	\end{align*}
	Set $\delta: \mathcal{C}_{n-1}(I,U)\rightarrow\mathcal{C}_{n}(I,U)$ satisfying  
	\begin{align*}
		\delta f_{t_{1},t_{2},\dots,t_{n} }=\sum_{i=1}^{n}(-1)^{i}f_{t_{1},\dots,\hat{t}_{i},\dots, t_{n}},
	\end{align*}
	where the argument $\hat{t}_{i}$ is omitted. 
	In particular, we note 
	\begin{align*}
		\delta f_{t,s}=f_{t}-f_{s},~~\text{and}~~ \delta g_{t,u,s}=g_{t,s}-g_{t,u}-g_{u,s}.
	\end{align*}
	Define 
	\begin{align*}
		\mathcal{C}_{1}^{\alpha}(I,U)&:=\{f\in \mathcal{C}_{1}(I,U): \|\delta f\|_{\alpha,U}=   \sup_{\substack{t_{1},t_{2}\in I\\ t_{1}\neq t_{2}   }  }\frac{\|f_{t_{2}}-f_{t_{1}}  \|_{U}}{|t_{2}-t_{1}|^{\alpha}}
  <\infty\},\\
		\mathcal{C}_{2}^{\alpha}(I,U)&:=\{f\in \mathcal{C}_{2}(I,U): \|f\|_{\alpha,U}<\infty\},\\
		\mathcal{C}_{3}^{\alpha_{1},\alpha_{2}}(I,U)&:=\{f\in \mathcal{C}_{3}(I,U): \|f\|_{\alpha_{1},\alpha_{2},U}<\infty\},
	\end{align*}
	Here, $\mathcal{C}_{1}^{\alpha}(I,U)$ denotes the space of all $\alpha$-H\"older continuous functions.
	We endow space $\mathcal{C}_{1}(I,U)$ with norm $\|f\|_{\infty,U}=\sup_{t\in I}\|f_{t}\|_{U}$, as well as space $\mathcal{C}_{1}^{\alpha}(I,U)$ with norm $  \|f\|_{\alpha,U}=\|f\|_{\infty,U}+\|\delta f\|_{\alpha,U}$. If $U=\mathcal{B}_{\theta}$, we write $\|\delta f\|_{\alpha, \theta}$ as shorthand. Similar simplicities are also used later.  Moreover, set $\mathcal{C}^{0,\alpha}_{\theta}(I):=\mathcal{C}_{1}(I,\mathcal{B}_{\theta})\cap\mathcal{C}_{1}^{\alpha}(I,\mathcal{B}_{\theta-\alpha}),$ and $\mathcal{C}^{\alpha,2\alpha}_{\theta}(I):=\mathcal{C}_{2}^{\alpha}(I,\mathcal{B}_{\theta-\alpha})\cap\mathcal{C}_{2}^{2\alpha}(I,\mathcal{B}_{\theta-2\alpha})$. And, interpolation inequality 
	(\ref{eq053}) implies
	\begin{align}\label{eq352}
		\|f_{t,s}\|_{\theta-\theta_{1}}\lesssim (\|f\|_{\alpha,\theta-\alpha}+\|f\|_{2\alpha,\theta-2\alpha})(t-s)^{\theta_{1}},
	\end{align}
	for $f\in\mathcal{C}^{\alpha,2\alpha}_{\theta}(I)$, $\theta_{1}\in[\alpha,2\alpha]$.


	Next, let us introduce delayed rough path, while the definition of rough path without delay provided in Appendix \ref{A.1}.  We remark that 
	delayed rough path for ODEs was proposed in \cite{MR2453555}, while the version presented here is for PDEs. 
	  
	\begin{Definition}[Delayed $\alpha$-Hölder  path] \label{def1}
		Let $X:\mathbb{R}\rightarrow \mathbb{R}^{d}$ be a locally $\alpha$-H\"older path. For a compact interval $I \subset \mathbb{R}$, set  $I_{r}:=\{e\in \mathbb{R}: e\in I ~\text{or}~  e+r\in I\}$. For $\alpha\in (\frac{1}{3},\frac{1}{2}]$, we call a triple  $\bar{\mathbf{X}}=(X, \mathbb{X},\mathbb{X}(-r))$  $\alpha$-Hölder rough path with delay $r>0$, if $X\in \mathcal{C}^{\alpha}_{1}(I_{r}, \mathbb{R}^d)$, $\mathbb{X}\in \mathcal{C}^{2 \alpha}_{2}\left(I_{r},\mathbb{R}^d \times \mathbb{R}^d\right)$ and  $\mathbb{X}(-r)\in \mathcal{C}^{2 \alpha}_{2}\left(I,\mathbb{R}^d \times \mathbb{R}^d\right)$ such that the following  Chen's relations hold:
		\begin{align*}
			\mathbb{X}^{i,j}_{t,s}-\mathbb{X}^{i,j}_{u,s}-\mathbb{X}^{i,j}_{t,u}&=X_{u,s}^{i}\otimes X_{t,u}^{j},\\
			\mathbb{X}(-r)^{i,j}_{t,s}-\mathbb{X}(-r)^{i,j}_{u,s}-\mathbb{X}(-r)^{i,j}_{t,u}&=\delta X^{i}_{u-r,s-r}\otimes \delta X^{j}_{t,u},
		\end{align*}
		for $t,u,s \in \Delta_{3,I}$, and $1\leq i, j\leq d$.
	\end{Definition}
	
	Compare to $\alpha$-H\"{o}lder path $\mathbf{X}=(X, \mathbb{X})$ as given in Definition \ref{A.11},  the additional L\'{e}vy area component $\mathbb{X}(-r)$ is considered in the delayed version $\mathbf{\bar{X}}$.  
	
	Denote  by $\bar{\mathcal{C}}^{\alpha}(I,\mathbb{R}^d)$ the space of delayed $\alpha$-H\"older rough paths  $\bar{\mathbf{X}}=(X, \mathbb{X},\mathbb{X}(-r))$.
	For two delayed $\alpha$-Hölder rough paths $\bar{\mathbf{X}}, \bar{\mathbf{Y}}\in \bar{\mathcal{C}}^{\alpha}(I, \mathbb{R}^d)$,  the metric of $\bar{\mathcal{C}}^{\alpha}(I, \mathbb{R}^d)$ is defined by 
	\begin{align*}
		\rho_{\alpha,I}(\bar{\mathbf{X}},\bar{\mathbf{Y}})=\|\delta(X-Y)\|_{\alpha,\mathbb{R}^{d}}+\|\mathbb{X}-\mathbb{Y}\|_{2\alpha,\mathbb{R}^{d}\times R^{d}}+\|\mathbb{X}(-r)-\mathbb{Y}(-r)\|_{2\alpha,\mathbb{R}^{d}\times \mathbb{R}^{d}}.
	\end{align*}
	Set $\rho_{\alpha,I}(\bar{\mathbf{X}})=	\rho_{\alpha,I}(\bar{\mathbf{X}},0)$.
	
	
	\begin{Definition}[Delayed controlled  path]\label{def002}
		Let  $\bar{\mathbf{X}}=(X, \mathbb{X},\mathbb{X}(-r))\in\bar{\mathcal{C}}^{\alpha}(I, \mathbb{R}^d)$ with $r>0$, $\alpha\in(\frac{1}{3},\frac{1}{2}]$. We call a triple $(y,y',\bar{y}')$  a delayed controlled  path based on $\bar{\mathbf{X}}$ if $y\in\mathcal{C}_{1}(I,\mathcal{B}_{\theta})$ and $y',\bar{y}'\in (\mathcal{C}^{0,\alpha}_{\theta-\alpha}(I))^{d}$  for $\theta\in\mathbb{R}$,  such that 
		\begin{align*}
			\bar{R}^{y}_{t,s}:&=\delta y_{t,s}-y'_{s}\cdot \delta X_{t,s}-\bar{y}'_{s}\cdot \delta X_{t-r,s-r}\\&=\delta y_{t,s}-\sum^{d}_{i=1}y'^{,i}_{s}\delta X^{i}_{t,s}-\sum^{d}_{i=1}\bar{y}'^{,i}_{s}\delta X^{i}_{t-r,s-r},~\text{for}~t,s \in \Delta_{2,I},
		\end{align*}
		belongs to $\mathcal{C}_{\theta}^{\alpha,2\alpha}(I)$.
	\end{Definition}

	Denote by $\mathcal{D}^{2\alpha}_{\bar{\mathbf{X}},\theta}(I)$ the space of delayed controlled rough paths based on $\bar{\mathbf{X}}$ with $I$, and endow the norm:
	\begin{align*}
		\|y,y',\bar{y}'\|_{\bar{\mathbf{X}},2\alpha,\theta}:&=\|y\|_{\infty,\theta}+\|y'\|_{\infty,\theta-\alpha}+\|y'\|_{\alpha,\theta-2\alpha}+\|\bar{y}'\|_{\infty,\theta-\alpha}+\|\bar{y}'\|_{\alpha,\theta-2\alpha}\\
		&~~~+\|\bar{R}^{y}\|_{\alpha,\theta-\alpha}+\|\bar{R}^{y}\|_{2\alpha,\theta-2\alpha}.
	\end{align*}
	Obviously, a controlled path is also  a delayed controlled path.
	Here, we would like to remark that
	$y$ in Definition \ref{def002}, has the H\"{o}lder regularity with respect to time. In fact, it is easy to check 
	\begin{align}
		\|\delta y\|_{\alpha,\theta-\alpha}&\leq (\|y'\|_{\infty,\theta-\alpha}+\|\bar{y}'\|_{\infty,\theta-\alpha}  )\|\delta X\|_{\alpha, \mathbb{R}^{d}}+  	\|\delta \bar{R}^{y}\|_{\alpha,\theta-\alpha}\notag\\
		&\leq (1+\rho_{\alpha, I_{-r}}(\bar{\mathbf{X}})) \|y,y',\bar{y}'\|_{\bar{\mathbf{X}},2\alpha,\theta} \label{eq103}
	\end{align}

	For two delayed rough path $(y,y',\bar{y}')$ and $(z,z',\bar{z}')$, we introduce the distance:
	\begin{align*}
		\rho_{2\tilde{\alpha},2\alpha, \theta}(y,z):&= \|y-z\|_{\infty,\theta}+\|y'-z'\|_{\infty,\theta-\alpha}+\|\delta(y'-z')\|_{\tilde{\alpha},\theta-2\alpha}+\|\bar{y}'-\bar{z}'\|_{\infty,\theta-\alpha}\\
		&~~~~+\|\delta(\bar{y}'-\bar{z}')\|_{\tilde{\alpha},\theta-2\alpha}+\|\bar{R}^{y}-\bar{R}^{z}\|_{\tilde{\alpha},\theta-\alpha}+ \|\bar{R}^{y}-\bar{R}^{z}\|_{2\tilde{\alpha},\theta-2\alpha},
	\end{align*}
	where $\tilde{\alpha},\alpha,\theta \in\mathbb{R}$.

	\section{Basic lemmas}
	In this section, we will show basic properties for delayed controlled path, and provide the definition of rough convolution. For the convenience  of reading, we present the proofs of these lemmas in the Appendix \ref{A.1}.
	
	Our first task is to prove that two controlled  paths composited with regular functions can become a delayed controlled   path. Compared with  previous results \cite{MR4299812,MR2453555},  our version  not only incorporates suitable space-time regularity, but also take account into  the effect of delay.
	\begin{Lemma}\label{lem001}
		Let $G\in \mathcal{C}^{2}_{\theta-2\alpha,-\sigma}(\mathcal{B}^{2},\mathcal{B})$ with $\theta\in\mathbb{R}$, $\alpha\in(\frac{1}{3},\frac{1}{2}]$, $\sigma\geq 0$, $T>0$, $r>0$. Let $(y,y')\in \mathcal{D}^{2\alpha }_{\mathbf{X}, \theta}(I_{1})$,  $(z,z')\in \mathcal{D}^{2\alpha }_{\mathbf{X}, \theta}(I_{2})$ with $I_{1}=[0,T]$ and $I_{2}=[-r,T-r]$. Set 
		\begin{align*}
			(m_{t}, m'_{t},\bar{m}'_{t}):=(G(y_{t},z_{t-r}), D_{x_{1}}G(y_{t},z_{t-r})\circ y'_{t},  D_{x_{2}}G(y_{t}, z_{t-r})\circ z'_{t-r}),
		\end{align*}
		where $D_{x_{i}}$ denotes the derivative of the $i$-th argument.  Then, we have $(m, m',\bar{m}')\in \mathcal{D}^{2\alpha}_{\bar{\mathbf{X}},\theta-\sigma}(I_{1})$, and 
		\begin{align}
			\|(m, m',\bar{m}')\|_{  \bar{\mathbf{X}}, 2\alpha,\theta-\sigma}&\lesssim_{G} (1+\rho_{\alpha, I_{1}}(\mathbf{X})+\rho_{\alpha, I_{2}}(\mathbf{X}) )^{2}\times(1+\|y,y'\|_{\mathbf{X},2\alpha,\theta}+
			\|z,z'\|_{\mathbf{X},2\alpha,\theta})^{2}.\label{eq101}
		\end{align} 
	\end{Lemma}
	
		\begin{Corollary}\label{cor001}
			Under the same settings of Lemma $\ref{lem001}$, we further assume $y'_{t}=G(y_{t},z_{t-r})$. Moreover, assume for  $y\in \mathcal{B}_{\theta-\alpha}$,   
			the derivatives  of 
			\begin{align*}
				D_{x_{1}}G(\cdot,y)\circ G(\cdot,y): \mathcal{B}_{\theta-\alpha}\rightarrow \mathcal{B}_{\theta-2\alpha-\sigma}
			\end{align*}
			and 
			\begin{align*}
				D_{x_{1}}G(y,\cdot)\circ G(y,\cdot): \mathcal{B}_{\theta-\alpha}\rightarrow \mathcal{B}_{\theta-2\alpha-\sigma}
			\end{align*}
			are both bounded. Then 	\begin{align*}
				\|(m, m',\bar{m}')\|_{  \bar{\mathbf{X}}, 2\alpha,\theta-\sigma}&\lesssim_{G} (1+\rho_{\alpha, I_{1}}(\mathbf{X})+\rho_{\alpha, I_{2}}(\mathbf{X}) )^{2}\times(1+\|y,y'\|_{\mathbf{X},2\alpha,\theta}\\
				&~~~+\|y,y'\|_{\mathbf{X},2\alpha,\theta}\|z,z'\|_{\mathbf{X},2\alpha,\theta}+ 
				\|z,z'\|^{2}_{\mathbf{X},2\alpha,\theta}).
			\end{align*} 
		\end{Corollary}

	Based on Lemma \ref{lem001}, we further have the following result in case of $G$ with better regularity.
	
	\begin{Lemma}\label{lem005}
		Under the same settings of Lemma $\ref{lem001}$, we further assume $G\in \mathcal{C}^{3}_{\theta-2\alpha,-\sigma}(\mathcal{B}^{2},\mathcal{B})$. Let $(u,u')\in \mathcal{D}^{2\alpha }_{\mathbf{X}, \theta}(I_{1})$,  $(v,v')\in \mathcal{D}^{2\alpha }_{\mathbf{X}, \theta}(I_{2})$. Set 
		\begin{align*}
			(l_{t}, l'_{t},\bar{l}'_{t}):=(G(u_{t},v_{t-r}), D_{x_{1}}G(u_{t},v_{t-r})\circ u'_{t},  D_{x_{2}}G(u_{t}, v_{t-r})\circ v'_{t-r}).
		\end{align*}
		Then, we have 
		\begin{align*}
			&~~~\|(m-l, m'-l',\bar{m}'-\bar{l}')\|_{\bar{\mathbf{X}}, 2\alpha,\theta-\sigma}\\
			&\lesssim_{G} (1+\rho_{\alpha, I_{1}}(\mathbf{X})+\rho_{\alpha, I_{2}}(\mathbf{X}) )^{2} (\|y-u,y'-u'\|_{\mathbf{X},2\alpha,\theta}+\|z-v,z'-v'\|_{\mathbf{X},2\alpha,\theta}    )\\
			&~~~\times(1+\|y,y'\|_{\mathbf{X},2\alpha,\theta}+
			\|z,z'\|_{\mathbf{X},2\alpha,\theta}+\|u,u'\|_{\mathbf{X},2\alpha,\theta}+\|v,v'\|_{\mathbf{X},2\alpha,\theta})^{2}.
		\end{align*}
	\end{Lemma}

		We will define the integral  of stochastic convolution with delay term.  In order to achieve this goal,  we  introduce some notations and recall sewing lemma (Theorem 4.1 in \cite{MR4299812}) as follows.
	
	Define 
	\begin{align*}
		\mathcal{J}^{\alpha}_{\theta}(I):&=\{f\in \mathcal{C}_{2}(I,\mathcal{B}_{\theta}):  f_{t,s}=f^{1}_{t,s}+f^{2}_{t,s}~\text{and}~ \delta f_{t,u,s}=f^{3}_{t,u,s}+f^{4}_{t,u,s}, ~\text{for}~s,u,t\in I,\\
		&~~~\text{where}~f^{1}\in \mathcal{C}_{2}^{\alpha}(I,\mathcal{B}_{\theta}),  f^{2}\in\mathcal{C}^{2\alpha}_{2}(I,\mathcal{B}_{\theta-\alpha}), f^{3}\in \mathcal{C}_{3}^{2\alpha,\alpha}(I,\mathcal{B}_{\theta-2\alpha}), f^{4}\in\mathcal{C}_{3}^{\alpha,2\alpha}(I,\mathcal{B}_{\theta-2\alpha}).
		\}
	\end{align*}
	The space $\mathcal{J}^{\alpha}_{\theta}$ is endowed with the norm:
	\begin{align*}
		\|f\|_{\mathcal{J}^{\alpha}_{\theta}}:= \inf_{f^{1},f^{2},f^{3},f^{4}}(\|f^{1}\|_{\alpha,\theta}+\|f^{2}\|_{2\alpha,\theta-\alpha}+\|f^{3}\|_{2\alpha,\alpha,\theta-2\alpha}+\|f^{4}\|_{\alpha,2\alpha,\theta-2\alpha}).
	\end{align*}
    \begin{Lemma}[Sewing lemma]\label{lem002}
    	Let  $I=[0,T]\in\mathbb{R}$, $\alpha\in (\frac{1}{3},\frac{1}{2}]$, $\theta\in\mathbb{R}$. There exists a unique continuous linear map $\mathcal{K}:\mathcal{J}^{\alpha}_{\theta}(I)\rightarrow \mathcal{C}^{0,\alpha}_{\theta}(I)$  such that the followings hold:
    	
    	$(i)$ The map $\mathcal{K}$ is defined by
    	\begin{align*}
    		\mathcal{K}_{0}(f):&=0,~~~~~~~~~~~~~~~~~~~~~~f\in \mathcal{J}^{\alpha}_{\theta}(I),   \\
    		\mathcal{K}_{t}(f):&=\lim_{|\mathcal{P}|\rightarrow 0}\sum_{[u,v]\in\mathcal{P}}S_{t-u}f_{v,u},~f\in \mathcal{J}^{\alpha}_{\theta}(I), t\in(0,T],
    	\end{align*}
    	where  $|\mathcal{P}|$ denotes the length of the largest element of a partition $\mathcal{P}$ of $[0, t]$, and such equality holds in sense of topology of $\mathcal{B}_{\theta-2\alpha}$.
    	
    	$(ii)$ For $0\leq \beta<3\alpha$, $t,s\in [0,T]$, we have 
    	\begin{align*}
    		\|\mathcal{K}_{t}(f)-S_{t,s}\mathcal{K}_{s}(f)-S_{t,s}f_{t,s}\|_{\theta-2\alpha+\beta}\lesssim\|f\|_{\mathcal{J}_{\theta}^{\alpha}}|t-s|^{3\alpha-\beta}.
    	\end{align*}
    \end{Lemma}

	The next lemma is devoted to the definition of  rough convolution with respect to delayed $\alpha$-H\"{o}lder path $\bar{\mathbf{X}}=(X,\mathbb{X},\mathbb{X}(-r))$. 
	
	\begin{Lemma}[Rough convolution]\label{lem003}
		Let $\bar{\mathbf{X}}=(X, \mathbb{X},\mathbb{X}(-r))\in \bar{\mathcal{C}}^{\alpha}(I,\mathbb{R}^d)$ with $I=[0,T]$, $\alpha\in (\frac{1}{3},\frac{1}{2}]$ and  $r>0$. Let delayed controlled path $(y^{i}, y^{i,\prime},{\bar{y}}^{i,\prime})$  based on $\bar{\mathbf{X}}$ belong to  $\mathcal{D}^{2\alpha}_{\bar{\mathbf{X}},\theta}(I)$ for $i=1,\cdots,d$, $\theta\in\mathbb{R}$.  Then, we have the  following results:
		
		$(i)$ For $t\in [0,T]$, the rough convolution 
		\begin{align*}
			\int^{t}_{0}S_{t,s} y_{s}\cdot\textup{d}\bar{\mathbf{X}}_{s}:&=\lim_{|\mathcal{P}|\rightarrow 0}\sum_{[u,v]\in\mathcal{P}} S_{t,u}[y_{u}\cdot\delta X_{v,u}+y^{\prime}_{u}:\mathbb{X}_{v,u}+   \bar{y}^{\prime}_{u}:\mathbb{X}(-r)_{v,u}],\\
			&=\lim_{|\mathcal{P}|\rightarrow 0}\sum_{[u,v]\in\mathcal{P}} S_{t,u}\big\{y_{u}\cdot\delta X_{v,u}+\sum^{d}_{i,j=1}[ y^{i,j,\prime}_{u}\mathbb{X}^{j,i}_{v,u}+   \bar{y}^{i,j,\prime}_{u}\mathbb{X}^{j,i}(-r)_{v,u}]\big\},
		\end{align*}
		belongs to $\mathcal{B}_{\theta-2\alpha}$, where $|\mathcal{P}|$ denotes the length of the largest element of a partition $\mathcal{P}$ of $[0, t]$, $(y^{i,j,\prime})$ and $(\bar{y}^{i,j,\prime})$  denotes the $j$-th element of $y^{i,\prime}$ and $ y^{i,\prime}$, respectively.
		
		$(ii)$ For $0\leq \beta<3\alpha$, $t,s\in [0,T]$, we have
		\begin{align*}
			&~~~\|\int^{t}_{s}S_{t,u} y_{u}\cdot\textup{d}\bar{\mathbf{X}}_{u}- S_{t,s}[y_{s}\cdot\delta X_{t,s}+y^{\prime}_{s}:\mathbb{X}_{t,s}+   \bar{y}^{\prime}_{s}:\mathbb{X}(-r)_{t,s}]\|_{\theta-2\alpha+\beta}\\
			&\lesssim  \|y\cdot\delta X+y^{\prime}:\mathbb{X}+   \bar{y}^{\prime}:\mathbb{X}(-r)\|_{\mathcal{J}^{\alpha}_{\theta} }(t-s)^{3\alpha-\beta}\\
			& \lesssim  \rho_{\alpha,I}(\bar{\mathbf{X}}) \|y,y',\bar{y}'\|_{\bar{\mathbf{X}},2\alpha,\theta}(t-s)^{3\alpha-\beta}.
		\end{align*} 
	\end{Lemma}

	Next, we proceed to prove the rough convolution $\int^{t}_{0}S_{t,s} y_{s}\cdot\textup{d}\bar{\mathbf{X}}_{s}$  is  a controlled path.

			\begin{Lemma}\label{lem006}
				Let $\bar{\mathbf{X}}=(X, \mathbb{X},\mathbb{X}(-r))\in \bar{\mathcal{C}}^{\alpha}(I,\mathbb{R}^d)$ with $I=[0,T]$, $T\leq 1$, $\alpha\in (\frac{1}{3},\frac{1}{2}]$, $r>0$,  $\tilde{\alpha}\in(0,\alpha]$, and $\sigma\in(0,\tilde{\alpha})$. Let delayed controlled path $(y^{i}, y^{i,\prime},{\bar{y}}^{i,\prime})$  based on $\bar{\mathbf{X}}$ belong to  $\mathcal{D}^{2\alpha}_{\bar{\mathbf{X}},\theta}(I)$ for  $\theta\in\mathbb{R}$, $i=1,\cdots,d$. Set
				\begin{align*}
					\zeta_{t}= 	\int^{t}_{0}S_{t,s} y_{s}\cdot\textup{d}\bar{\mathbf{X}}_{s}.
				\end{align*}
				Then, the followings hold:
				
				$(i)$ $(\zeta,\zeta')$ is controlled path based  $\mathbf{X}$ belong to  $\mathcal{D}^{2\alpha}_{\mathbf{X},\theta+\sigma}(I)$ with $\zeta'=y$.
				
				$(ii)$ $\|\zeta,\zeta'\|_{\mathbf{X},2\tilde{\alpha}, \theta+\sigma}\lesssim (1+\rho_{\alpha, I}(\bar{\mathbf{X}}))  \|y,y',\bar{y}'\|_{\bar{\mathbf{X}},2\tilde{\alpha},\theta}T^{\lambda_{0}}+\|y_{0}\|_{\theta},$
				where $\lambda_{0}:=(\alpha-\tilde{\alpha})\wedge(\tilde{\alpha}-\sigma)$.
				
				$(iii)$ $\|\zeta,\zeta'\|_{\mathbf{X},2\alpha, \theta+\sigma}\lesssim (1+\rho_{\alpha, I}(\bar{\mathbf{X}}))[ \|y,y',\bar{y}'\|_{\bar{\mathbf{X}},2\alpha,\theta}T^{\alpha-\sigma}+\|y'_{0}\|_{\theta-\alpha}+\|\bar{y}'_{0}\|_{\theta-\alpha}]+\|y_{0}\|_{\theta}$.
			\end{Lemma}

			The following lemma is devoted to  the stability of the rough convolution.  		
			\begin{Lemma}\label{lem777}
				Under the same settings of Lemma $\ref{lem006}$, we further assume   
				$\bar{\mathbf{Y}}=(Y, \mathbb{Y},\mathbb{Y}(-r))\in \bar{\mathcal{C}}^{\alpha}(I,\mathbb{R}^d)$. Let delayed controlled path $(z^{i}, z^{i,\prime},{\bar{z}}^{i,\prime})$  based on $\bar{\mathbf{Y}}$ belong to  $\mathcal{D}^{2\alpha}_{\bar{\mathbf{Y}},\theta}(I)$ for  $\theta\in\mathbb{R}$, $i=1,\cdots,d$.   Set 
				\begin{align*}
					\chi_{t}= 	\int^{t}_{0}S_{t,s} z_{s}\cdot\textup{d}\bar{\mathbf{Y}}_{s}.
				\end{align*}
				Taking $\tilde{\alpha}\in(\sigma,\alpha)$ such that $3\tilde{\alpha}-2\alpha-\sigma>0$, we have 
				\begin{align*}
					&\rho_{2\tilde{\alpha},2\alpha, \theta+\sigma}(\zeta,\chi)\lesssim _{C} \rho_{2\tilde{\alpha},2\alpha, \theta}(y,z)T^{\lambda}+\rho_{\alpha,I}(\bar{\mathbf{X}},\bar{\mathbf{Y}})+   \|y(0)-z(0)\|_{\theta},
				\end{align*}
				where $C$ depends on $\|y,y',\bar{y}'\|_{\bar{\mathbf{X}},2\alpha,\theta}, \|z,z',\bar{z}'\|_{\bar{\mathbf{Y}},2\alpha,\theta}, \bar{\mathbf{X}}$ and $\bar{\mathbf{Y}}$, $\lambda=\min\{\alpha-\tilde{\alpha},\frac{\tilde{\alpha}(\alpha-\sigma)}{\alpha}, 3\tilde{\alpha}-2\alpha-\sigma\}$.

			\end{Lemma}

			\section{Delay rough PDEs}

			With the help of previous preliminaries,  we now can consider the following delay rough PDEs: 
			\begin{align}
				\textup{d}y_{t}&=[A y_{t}+F(y_{t}, y_{t-r})]\textup{d}t+G(y_{t},y_{t-r})\cdot\textup{d}\bar{\mathbf{X}}_{t},\label{eq011}\\
				y_{t}&=\phi_{t}, ~~ t\in[-r,0].\label{eq012}
			\end{align}
			
			Recalling the linear operator $A$ can generate an analytic $C_{0}$-semigroup $\{S_{t}\}_{t\in \mathbb{R}^{+}}$ on $\mathcal{B}$, we focus on the mild solutions of (\ref{eq011})-(\ref{eq012}) as follows: 
			\begin{align}
				y_{t}&=S_{t,0}y_{0}+\int^{t}_{0}S_{t,s}F(y_{s},y_{s-r})\textup{d}{s}+\int^{t}_{0}S_{t,s}G(y_{s},y_{s-r})\cdot\textup{d}\bar{\mathbf{X}}_{s},\label{eq109}\\
				y_{t}&=\phi_{t}, ~~ t\in[-r,0].\label{eq110}
			\end{align}
			We will show the existence and stability of the solutions of   (\ref{eq109})-(\ref{eq110}) on any $[0,T]$ under the following  assumptions.
			
			$\mathbf{H} \mathbf{1}$:  Let $\mathbf{X}=(X, \mathbb{X})\in\mathcal{C}^{\alpha}([0,T], \mathbb{R}^d)$ and $\bar{\mathbf{X}}=(X, \mathbb{X},\mathbb{X}(-r))\in\bar{\mathcal{C}}^{\alpha}([0,T], \mathbb{R}^d)$ with  $T>0$, $r>0$, $\alpha\in(\frac{1}{3},\frac{1}{2}]$.

			$\mathbf{H} \mathbf{2}$:  $(\phi,\phi')$ is a controlled path based on $\mathbf{X}$ belonging to $\mathcal{D}^{2\tilde{\alpha}}_{\mathbf{X},\theta}([-r,0])\cap\mathcal{D}^{2\alpha}_{\mathbf{X},\theta}([-r,0]) $ with $\tilde{\alpha}\in(0,\alpha)$ and $\theta\in\mathbb{R}$.
			
			$\mathbf{H} \mathbf{3}$:  Nonlinear operator $F\in  \text{Lip}_{\theta,-\sigma_{1}}(\mathcal{B}^{2}, \mathcal{B})$ with $\sigma_{1}\in(0,1)$.
			
			$\mathbf{H} \mathbf{4}$:  Nonlinear operator $G=(G_{1},\cdots,G_{d})$, such that $G_{i}\in  \mathcal{C}^{3}_{\theta-2\alpha,-\sigma_{2}}(\mathcal{B}^{2}, \mathcal{B})$ with  $\sigma_{2}\in[0,\tilde{\alpha})$ for $i=1,\cdots,d$. Moreover, we assume that  for fixed $y\in \mathcal{B}_{\theta-\alpha}$, $1\leq i,j\leq d$,   
			the derivatives  of 
			\begin{align}
				D_{x_{1}}G_{i}(\cdot,y)\circ G_{j}(\cdot,y): \mathcal{B}_{\theta-\alpha}\rightarrow \mathcal{B}_{\theta-2\alpha-\sigma_{2}}\label{eq601}
			\end{align}
			and 
			\begin{align}
				D_{x_{1}}G_{i}(y,\cdot)\circ G_{j}(y,\cdot): \mathcal{B}_{\theta-\alpha}\rightarrow \mathcal{B}_{\theta-2\alpha-\sigma_{2}}\label{eq602}
			\end{align}
			are both bounded.
			\begin{Remark}	~~
				
				$(i)$ We would like to explain why we need both $\mathbf{X}$ and  $\bar{\mathbf{X}}$. On the one hand, we hope to prove $(y,y')$ is a controlled path based $\mathbf{X}$. On the other hand, we can guarantee that the rough integral is  well-defined  via $\bar{\mathbf{X}}$ as Lemma $\ref{lem003}$ shown. 
				
				$(ii)$ According to $\mathbf{H 3}$, for $\theta_{1}\geq \theta $, $x$, $y\in \mathcal{B}_{\theta_{1}}$, it is easy to derive 
				\begin{align*}
					\|F(x,y)\|_{\theta_{1}-\sigma_{1}}&\lesssim \|F(x,y)-F(0,0)\|_{\theta_{1}-\sigma_{1}}+ \|F(0,0)\|_{\theta_{1}-\sigma_{1}}\\
					&\lesssim_{F}\|x\|_{\theta_{1}}+\|y\|_{\theta_{1}}+1.
				\end{align*}
				
				$(iii)$ If we only want to need the local mild
				solution, then  $G_{i}\in  \mathcal{C}^{3}_{\theta-2\alpha,-\sigma_{2}}(\mathcal{B}^{2}, \mathcal{B})$ is enough for us in  $\mathbf{H 4}$. The extra conditions $(\ref{eq601})$ and $(\ref{eq602})$ can allow us to  to obtain for $i=1,\cdots,d$,
				\begin{align*}
					\|[D_{x_{1}}G_{i}(y_{1},z_{1})- D_{x_{1}}G_{i}(y_{2},z_{2})]\circ G_{i}(y_{1},z_{1})\|_{\theta-2\alpha-\sigma_{2}}\lesssim_{G} \|y_{1}-y_{2}\|_{\theta-\alpha}+ \|z_{1}-z_{2}\|_{\theta-\alpha},
				\end{align*}
				by which we can pursue the global mild solution. 
				Moreover, we  note $(\ref{eq601})$ and $(\ref{eq602})$ can be satisfied if $G$ is linear operator or bounded operator. Such assumptions also are used to investigate the global solutions for  rough PDEs \textup{\cite{MR4431448}} or delay rough PDEs \cite{MR3236092}. 
			\end{Remark}
			
			 Set $\hat{r}:=r\wedge 1$.  If $f\leq C g$, where positive constant $C$ we  are dependent of  $ \bar{\mathbf{X}}$, $F$  and  $G$, we write $f\lesssim g$ for simplicity in what follows.
			
			\subsection{Existence of solutions}
			We  focus on  (\ref{eq109})-(\ref{eq110}) on $[0,r]$ as start.
			In fact, if  we can prove that there exist the solutions of  (\ref{eq109})-(\ref{eq110})  on $[0,r]$, then the corresponding assertion also holds on any $[0,T]$ by iteration. Keep in mind,  we adopt $\|\cdot\|_{[a,b]}$ and $\rho_{[a,b]}$ to denote the norm and distance is with respect to the interval $[a,b]$.
			
			The sketch of our proof consists of three steps.
			We are going to show the existence of local mild solution of 
			(\ref{eq109})-(\ref{eq110}) in the first step. 
			The second one is devoted to a priori estimate
			of the solutions. Combining the results of the first two steps, we show the existence of global mild solution of 
			(\ref{eq109})-(\ref{eq110}) in the final step.

			{\bf \normalsize Step \uppercase\expandafter{\romannumeral 1}: Existence of local mild solution
			} 
			
			Since $(\phi,\phi')\in \mathcal{D}^{2\alpha}_{\mathbf{X},\theta}([-r,0])$ and $y_{0}=\phi_{0}$,   it is obvious to obtain 
			\begin{align*}
				(S_{\cdot,0}y_{0},0)\in\mathcal{D}^{2\alpha}_{\mathbf{X},\theta}([0,r]), 
			\end{align*}
			and 
			\begin{align}\label{eq117}
				\|(S_{\cdot,0}y_{0},0)\|_{\mathbf{X},2\alpha, \theta,[0,r]}\leq C_{1} \|y_{0}\|_{\theta},_.
			\end{align}
			
			\begin{Lemma}\label{lem007}
				If $(y,y')\in\mathcal{D}^{2\tilde{\alpha}}_{\mathbf{X},\theta}([0,\tilde{r}]) $ with  $\tilde{r}\in(0,\hat{r}]$, then we have 
				\begin{align*}
					(\int^{.}_{0}S_{.,s}F(y_{s},y_{s-r})\textup{d}{s},0)\in \mathcal{D}^{2\tilde{\alpha}}_{\mathbf{X},\theta}([0,\tilde{r}]),
				\end{align*} and
				\begin{align*}
					\|(\int^{.}_{0}S_{.,s}F(y_{s},\phi_{s-r})\textup{d}{s},0)\|_{\mathbf{X},2\tilde{\alpha}, \theta, [0,\tilde{r}]}\leq C_{2}(1+\|y\|_{\infty,\theta, [0,\tilde{r}]}+\|\phi\|_{\infty,\theta, [-r,0]})\tilde{r}^{\lambda_{1}},
				\end{align*} 
				where  $\lambda_{1}:=(1-2\tilde{\alpha})\wedge(1-\sigma_{1})$, $C_{2}$ depends on $F$. 
			\end{Lemma}
			\begin{proof}
				
				Clearly,  the Gubinelli derivative of the deterministic integral is $0$. We just need to estimate each norm of  $\int^{t}_{0}S_{t,s}F(y_{s},\phi_{s-r})\textup{d}{s}$.
				
				According to $\mathbf{H}\mathbf{3}$ and (\ref{eq050}),  we obtain 
				\begin{align*}
					\|\int^{t}_{0}S_{t,s}F(y_{s},\phi_{s-r})\textup{d}{s}\|_{\theta}
					&\lesssim \int^{t}_{0}(t-s)^{-\sigma_{1}}\|F(y_{s},\phi_{s-r})\|_{\theta-\sigma_{1}}  \textup{d}{s}\\
					&\lesssim t^{1-\sigma_{1}}(1+\|y\|_{\infty,\theta, [0,\tilde{r}]}+\|\phi\|_{\infty,\theta, [-r,0]}),
				\end{align*}
				which implies 
				\begin{align*}
					\|\int^{\cdot}_{0}S_{t,s}F(y_{s},\phi_{s-r})\textup{d}{s}\|_{\infty,\theta}\lesssim \tilde{r}^{1-\sigma_{1}}(1+\|y\|_{\infty,\theta, [0,\tilde{r}]}+\|\phi\|_{\infty,\theta, [-r,0]}).
				\end{align*}
				
				Similarly, for $i=1,2$, we have 
				\begin{align*}
					&~~~~\|\int^{t}_{0}S_{t,u}F(y_{u},\phi_{u-r})\textup{d}{u}-\int^{s}_{0}S_{s,u}F(y_{u},\phi_{u-r})\textup{d}{u}\|_{\theta-i\tilde{\alpha}}\\
					&\lesssim \|\int^{t}_{s}S_{t,u}F(y_{u},\phi_{u-r})\textup{d}{u}\|_{\theta-i\tilde{\alpha}}+ \|(S_{t,s}-\textup{id}) \int^{s}_{0}S_{s,u}F(y_{u},\phi_{u-r})\textup{d}{u}  \|_{\theta-i\tilde{\alpha}}\\
					&\lesssim (t-s)^{\lambda_{1}+i\tilde{\alpha} }(1+\|y\|_{\infty,\theta, [0,\tilde{r}]}+\|\phi\|_{\infty,\theta, [-r,0]})+(t-s)^{i\tilde{\alpha}}s^{1-\sigma_{1}}(1+\|y\|_{\infty,\theta, [0,\tilde{r}]}\\
					&~~~+\|\phi\|_{\infty,\theta, [-r,0]}).
				\end{align*}
				Furthermore, for $i=1,2$, we obtain  
				\begin{align*}
					\|\int^{\cdot}_{0}S_{\cdot,u}F(y_{u},\phi_{u-r})\textup{d}{u}\|_{i\tilde{\alpha},\theta-i\tilde{\alpha}}\lesssim \tilde{r}^{\lambda_{1}}(1+\|y\|_{\infty,\theta, [0,\tilde{r}]}+\|\phi\|_{\infty,\theta, [-r,0]}).
				\end{align*}
				
				Concluding previous estimates, we complete the proof.
			\end{proof}
			
			\begin{Lemma}\label{lem008}
				If $(y,y')\in\mathcal{D}^{2\tilde{\alpha}}_{\mathbf{X},\theta}([0,\tilde{r}]) $ with $\tilde{r}\in(0,\hat{r}]$, then we have 
				\begin{align*}
					(\int^{.}_{0}S_{\cdot,s}G(y_{s},\phi_{s-r})\cdot\textup{d}{\bar{\mathbf{X}}}_{s},G(y_{.},\phi_{.-r}))\in \mathcal{D}^{2\tilde{\alpha}}_{\mathbf{X},\theta}([0,\tilde{r}]),
				\end{align*}
				and
				\begin{align}
					&~~~~\|(\int^{.}_{0}S_{.,s}G(y_{s},\phi_{s-r})\cdot\textup{d}{\bar{\mathbf{X}}}_{s},G(y_{.},y_{.-r}))\|_{\mathbf{X}, 2\tilde{\alpha}, \theta, [0,\tilde{r}]}\notag \\
					&\leq C_{3} [(1+\|y,y'\|_{\mathbf{X},2\tilde{\alpha},\theta, [0,\tilde{r}]}+
					\|\phi,\phi'\|_{\mathbf{X},2\tilde{\alpha},\theta, [-r,0]})^{2}\tilde{r}^{\lambda_{2}}+\|y_{0}\|_{\theta}+\|\phi_{-r}\|_{\theta}], \label{eq377}
				\end{align}
				where $\lambda_{2}:=(\alpha-\tilde{\alpha})\wedge(\tilde{\alpha}-\sigma_{2})$, $C_{3}$ depends on   $ \bar{\mathbf{X}}$ and $G$.
			\end{Lemma}
			\begin{proof}
				It follows from Lemma \ref{lem001} and $\mathbf{H4}$ that for $i=1,\cdots, d$, 
				\begin{align*}
					( G^{i}(y_{.},\phi_{.-r}),    G^{i,\prime}(y_{.},\phi_{.-r}), \bar{G}^{i,\prime}(y_{.},\phi_{.-r})) \in \mathcal{D}^{2\tilde{\alpha}}_{\bar{\mathbf{X}},\theta-\sigma_{2}}([0,\tilde{r}]).
				\end{align*}  
				Then, applying $(ii)$ of Lemma \ref{lem006}, we obtain 
				\begin{align*}
					&~~~~\|(\int^{.}_{0}S_{.,s}G(y_{s},\phi_{s-r})\cdot\textup{d}{\bar{\mathbf{X}}},G(y_{.},\phi_{.-r}))\|_{\mathbf{X},2\tilde{\alpha}, \theta, [0,\tilde{r}]}\\
					&\lesssim \|G(y_{.},\phi_{.-r}),    G^{\prime}(y_{.},\phi_{.-r}), \bar{G}^{\prime}(y_{.},\phi_{.-r})\|_{ \bar{\mathbf{X}},2\tilde{\alpha},\theta-\sigma_{2},[0,\tilde{r}]}\tilde{r}^{\lambda_{2}}+\| G(y_{0},\phi_{-r})\|_{\theta-\sigma_{2}}\\
					&\lesssim (1+\|y,y'\|_{\mathbf{X},2\tilde{\alpha},\theta,[0,\tilde{r}]}+
					\|\phi,\phi'\|_{\mathbf{X},2\tilde{\alpha},\theta,[-r,0]})^{2} \tilde{r}^{\lambda_{2}}+\|y_{0}\|_{\theta}+\|\phi_{-r}\|_{\theta},
				\end{align*}
				where we use Lemma \ref{lem001} and $\mathbf{H4}$ again in the last inequality.
				
				The proof is completed after further calculations. 
			\end{proof}

			\begin{Lemma}\label{lem011}
				For $\tilde{\alpha}\in(\sigma_{2},\alpha)$, if $(y,y')\in\mathcal{D}^{2\tilde{\alpha}}_{\mathbf{X},\theta}([0,\tilde{r}])$ is the mild solution of $(\ref{eq109})$-$(\ref{eq110})$  on $[0,\tilde{r}]$	with $\tilde{r}\in(0,r]$ and $y_{t}'=G(y_{t},\phi_{t-r})$, then we have $(y,y')\in\mathcal{D}^{2\alpha}_{\mathbf{X},\theta}([0,\tilde{r}])$.
			\end{Lemma}
			\begin{proof}
				Recall 
				\begin{align*}
					y_{t}&=S_{t,0}y_{0}+\int^{t}_{0}S_{t,s}F(y_{s},\phi_{s-r})\textup{d}{s}+\int^{t}_{0}S_{t,s}G(y_{s},\phi_{s-r})\cdot\textup{d}\bar{\mathbf{X}}_{s}.
				\end{align*}
				
				Obviously, $\|y'\|_{\infty,\theta-\alpha}\lesssim \| y'\|_{\infty,\theta-\tilde{\alpha}}$.
				
				In order to show  $\|\delta y'\|_{\alpha,\theta-2\alpha, [0,\tilde{r}]}<\infty$, we prove $\|\delta y\|_{\alpha,\theta-\alpha, [0,\tilde{r}]}<\infty$ firstly.
				
				For $i=1,2$, notice  
				\begin{align*}
					\|S_{t,0}y_{0}-S_{s,0}y_{0}\|_{\theta-i\alpha}=\|S_{s,0}(S_{t,s}-id)y_{0}\|_{\theta-i\alpha}\lesssim (t-s)^{i\alpha}\|y_{0}\|_{\theta},
				\end{align*}
				which implies 
				\begin{align*}
					\|\delta S_{\cdot,0}y_{0}\|_{i\alpha,\theta-i\alpha, [0,\tilde{r}]}\lesssim \|y_{0}\|_{\theta}.
				\end{align*}
				
				As the proof of Lemma \ref{lem007} shown, we also know 
				\begin{align*}
					\|\delta \int^{\cdot}_{0}S_{\cdot,s}F(y_{s},\phi_{s-r})\textup{d}{s}\|_{i\alpha,\theta-i\alpha, [0,\tilde{r}]} \lesssim (1+ \|y\|_{\infty,\theta, [0,\tilde{r}]}+\|\phi\|_{\infty,\theta,[-r,0]}).
				\end{align*} 
				
				Rewrite 
				\begin{align*}	&~~~~\int^{t}_{0}S_{t,u}G(y_{u},\phi_{u-r})\cdot\textup{d}\bar{\mathbf{X}}_{u}-\int^{s}_{0}S_{s,u}G(y_{u},\phi_{u-r})\cdot\textup{d}\bar{\mathbf{X}}_{u}\\
					&=\sum_{j=1}^{3}\mathcal{N}^{j}_{t,s},
				\end{align*}
				where
				\begin{align*}
					\mathcal{N}^{1}_{t,s}&=\int^{t}_{s}S_{t,u}G(y_{u},\phi_{u-r})\cdot\textup{d}\bar{\mathbf{X}}_{u}-S_{t,s}[G(y_{s},\phi_{s-r})\cdot\delta X_{t,s}+G^{\prime}(y_{s},\phi_{s-r}):\mathbb{X}_{t,s}\\
					&~~~+\bar{G}^{\prime}(y_{s},\phi_{s-r}):\mathbb{X}(-r)_{t,0}],\\
					\mathcal{N}^{2}_{t,s}&=S_{t,s}[G(y_{s},\phi_{s-r})\cdot\delta X_{t,s}+G^{\prime}(y_{s},\phi_{s-r}):\mathbb{X}_{t,s}+\bar{G}^{\prime}(y_{0},\phi_{-r}):\mathbb{X}(-r)_{t,0}],\\
					\mathcal{N}^{3}_{t,s}&=(S_{t,s}-id)\int^{s}_{0}S_{s,u}G(y_{u},\phi_{u-r})\cdot\textup{d}\bar{\mathbf{X}}_{u}.
				\end{align*}
				
				By Lemma \ref{lem003}, we derive 
				\begin{align*}
					\|\mathcal{N}^{1}_{t,s}\|_{\theta-i\alpha}\lesssim \| G(y_{\cdot},\phi_{\cdot-r}),G'(y_{\cdot},\phi_{\cdot-r}), \bar{G}'(y_{\cdot},\phi_{\cdot-r})\|_{\bar{\mathbf{X}},2\tilde{\alpha},\theta-\sigma_{2},[0,\tilde{r}]}(t-s)^{i\alpha+\tilde{\alpha}-\sigma_{2}}.
				\end{align*}
				From  Lemma \ref{lem001},  we further have 
				\begin{align*}
					\|\mathcal{N}^{1}\|_{i\alpha,\theta-i\alpha, [0,\tilde{r}]}&\lesssim \| G(y_{\cdot},\phi_{\cdot-r}),G'(y_{\cdot},\phi_{\cdot-r}), \bar{G}'(y_{\cdot},\phi_{\cdot-r})\|_{\bar{\mathbf{X}},2\tilde{\alpha},\theta-\sigma_{2}, [0,\tilde{r}]}r^{\tilde{\alpha}-\sigma_{2}}\\
					&\lesssim(1+\|y,y'\|_{\mathbf{X},2\tilde{\alpha},\theta,[0,\tilde{r}]}+
					\|\phi,\phi'\|_{\mathbf{X},2\tilde{\alpha} ,\theta,[-r,0]})^{2}r^{\tilde{\alpha}-\sigma_{2}}.
				\end{align*}
				Using Lemma \ref{lem001} again, 
				we obtain 
				\begin{align*}
					&\|S_{t,s}G(y_{s},\phi_{s-r})\cdot\delta X_{t,s}\|_{\theta-\alpha}\lesssim (\|y_{s}\|_{\theta}+\|\phi_{s}\|_{\theta})(t-s)^{\alpha},\\
					&\|S_{t,s}G'(y_{s},\phi_{s-r}):\mathbb{X}_{t,s}\|_{\theta-i\alpha}\lesssim (\|y'_{s}\|_{\theta-\alpha})(t-s)^{i\alpha},\\
					&\|S_{t,s}\bar{G}'(y_{s},\phi_{s-r}):\mathbb{X}_{t,s}\|_{\theta-i\alpha}\lesssim (\|\phi'_{s-r}\|_{\theta-\alpha})(t-s)^{i\alpha},
				\end{align*}
				from which we have 
				\begin{align*}
					\|\mathcal{N}^{2}\|_{\alpha,\theta-\alpha, [0,\tilde{r}]}\lesssim \|y,y'\|_{\mathbf{X},2\tilde{\alpha},\theta,[0,\tilde{r}]}+
					\|\phi,\phi'\|_{\mathbf{X},2\tilde{\alpha},\theta,[-r,0]}.
				\end{align*}

				For the last term, we rewrite 
				\begin{align*}
					\mathcal{N}^{3}_{t,s}=\sum_{j=1}^{4}\mathcal{N}^{3,j}_{t,s},
				\end{align*}
				where
				\begin{align*}
					\mathcal{N}^{3,1}_{t,s}&=(S_{t,s}-id)S_{s,0}G(y_{0},\phi_{-r})\cdot\delta X_{s,0},\\
					\mathcal{N}^{3,2}_{t,s}&=(S_{t,s}-id)S_{s,0}G^{\prime}(y_{0},\phi_{-r}):\mathbb{X}_{s,0},\\
					\mathcal{N}^{3,3}_{t,s}&=(S_{t,s}-id)S_{s,0}\bar{G}^{\prime}(y_{0},\phi_{-r}):\mathbb{X}(-r)_{s,0},\\
					\mathcal{N}^{3,4}_{t,s}&=(S_{t,s}-id)\int^{s}_{0}S_{s,u}G(y_{u},\phi_{u-r})\cdot\textup{d}\bar{\mathbf{X}}_{u}-S_{s,0}[G(y_{0},\phi_{-r})\cdot\delta X_{s,0}\\&~~~+G^{\prime}(y_{0},\phi_{-r}):\mathbb{X}_{s,0}
					+\bar{G}^{\prime}(y_{0},\phi_{-r}):\mathbb{X}(-r)_{s,0}].
				\end{align*}
				Note
				\begin{align*}
					\|\mathcal{N}^{3,1}_{t,s}\|_{\theta-i\alpha}&\lesssim (t-s)^{i\alpha}\|S_{s,0}G(y_{0},\phi_{-r})\|_{\theta}s^{\alpha}\\\
					&\lesssim (t-s)^{i\alpha}s^{\alpha-\sigma_{2}} \|G(y_{0},\phi_{-r})\|_{\theta-\sigma_{2}}\\
					&\lesssim (t-s)^{i\alpha}s^{\alpha-\sigma_{2}}(\|y_{0}\|_{\theta}+\|\phi_{-r}\|_{\theta}).
				\end{align*}
				Similarly, we also obtain 
				\begin{align*}
					\|\mathcal{N}^{3,2}_{t,s}\|_{\theta-i\alpha}&\lesssim (t-s)^{i\alpha}\|S_{s,0}G'(y_{0},\phi_{-r})\|_{\theta}s^{2\alpha}\\\
					&\lesssim (t-s)^{i\alpha}s^{2\alpha-\tilde{\alpha}-\sigma_{2}} \|y'_{0}\|_{\theta-\tilde{\alpha}},
				\end{align*}
				and 
				\begin{align*}
					\|\mathcal{N}^{3,3}_{t,s}\|_{\theta-i\alpha}&\lesssim (t-s)^{i\alpha}\|S_{s,0}\bar{G}'(y_{0},\phi_{-r})\|_{\theta}s^{2\alpha}\\\
					&\lesssim (t-s)^{i\alpha}s^{\alpha-\sigma_{2}} \|\phi'_{-r}\|_{\theta-\alpha}.
				\end{align*}
				And, choosing $\beta\in (\sigma_{2}+2\tilde{\alpha},3\tilde{\alpha})$, we note
				\begin{align*}
					\|\mathcal{N}^{3,4}_{t,s}\|_{\theta-i\alpha}&\lesssim (t-s)^{\beta-\sigma_{2}-2\tilde{\alpha}+i\alpha}\|\int^{s}_{0}S_{s,u}G(y_{u},\phi_{u-r})\cdot\textup{d}\bar{\mathbf{X}}_{u}-S_{s,0}[G(y_{0},\phi_{-r})\cdot\delta X_{s,0}\\&~~~+G^{\prime}(y_{0},\phi_{-r}):\mathbb{X}_{s,0}
					+\bar{G}^{\prime}(y_{0},\phi_{-r}):\mathbb{X}(-r)_{s,0}]\|_{\theta-\sigma_{2}-2\tilde{\alpha}+\beta}\\
					&\lesssim (t-s)^{\beta-\sigma_{2}-2\tilde{\alpha}+i\alpha}\| G(y_{\cdot},\phi_{\cdot-r}),G'(y_{\cdot},\phi_{\cdot-r}), \bar{G}'(y_{\cdot},\phi_{\cdot-r})\|_{\bar{\mathbf{X}},2\tilde{\alpha},\theta-\sigma_{2},[0,\tilde{r}]}\\
					&\lesssim (t-s)^{\beta-\sigma_{2}-2\tilde{\alpha}+i\alpha}(1+\|y,y'\|_{\mathbf{X},2\tilde{\alpha},\theta,[0,\tilde{r}]}+
					\|\phi,\phi'\|_{\mathbf{X},2\tilde{\alpha},\theta,[-r,0]})^{2},
				\end{align*}
				where we use Lemma \ref{lem003} in the second equality, and Lemmas \ref{lem001} in the last equality.
				
				Combining previous estimates, we have obtained $\|\delta y\|_{\alpha,\theta-\alpha, [0,\tilde{r}]}<\infty$. 
				In particular,  we also obtain 
				\begin{align}\label{eq501}
					\sup_{t,s\in[0,\tilde{r}]}\frac{\|y_{t}-y_{s}-S_{t,s}G(y_{s},\phi_{s-r})\cdot\delta X_{t,s}\|_{\theta-2\alpha}}{(t-s)^{2\alpha}}<\infty.
				\end{align} 
				Then,	we can derive 
				\begin{align*}
					\|\delta y'_{t,s}\|_{\theta-2\alpha}&=\|G(y_{t},\phi_{t-r})-G(y_{s},\phi_{s-r})\|_{\theta-2\alpha}\\
					&\lesssim \|y_{t}-y_{s}\|_{\theta+\sigma_{2}-2\alpha}+ \|\phi_{t-r}-\phi_{s-r}\|_{\theta+\sigma_{2}-2\alpha}\\
					&\lesssim(t-s)^{\alpha} (\|\delta y\|_{\alpha,\theta-\alpha, [0,\tilde{r}]}+ \|\delta \phi\|_{\alpha, \theta-\alpha, [-r,0]}),
				\end{align*}
				which implies 
				$\|\delta y'\|_{\alpha, \theta-2\alpha, [0,\tilde{r}]}<\infty$. 
				
				It is turn to consider 
				$R^{y}_{t,s}=y_{t}-y_{s}-y'_{s}\cdot \delta X_{t,s}$. It is easy to obtain 
				\begin{align*}
					\|R^{y}_{t,s}\|_{\theta-\alpha}\lesssim \|\delta y\|_{\alpha,\theta-\alpha, [0,\tilde{r}]}(t-s)^{\alpha}+\|y'\|_{\infty,\theta-\alpha, [0,\tilde{r}]}(t-s)^{\alpha}.
				\end{align*}
				Then, using (\ref{eq501}), we have 
				\begin{align*}
					\|R^{y}_{t,s}\|_{\theta-2\alpha}&\lesssim \|y_{t}-y_{s}-S_{t,s}G(y_{s},\phi_{s-r})\cdot\delta X_{t,s}\|_{\theta-2\alpha}\\
					&~~~+\| (S_{t,s}-id)G(y_{s},\phi_{s-r})\cdot\delta X_{t,s}\|_{\theta-2\alpha}\\
					&\lesssim (t-s)^{2\alpha}.
				\end{align*}
				
				Collecting the above estimates, we have proven $(y,y')\in\mathcal{D}^{2\alpha}_{\mathbf{X},\theta}([0,\tilde{r}])$.
			\end{proof}

			Next Theorem is devoted to the local  mild solution of (\ref{eq109})$-$(\ref{eq110}).
			
			\begin{Theorem}\label{the001}
				There exists $\tilde{r}\in(0,r]$ such that $ (y,y')\in\mathcal{D}^{2\tilde{\alpha}}_{\mathbf{X},\theta}([0,\hat{r}])\cap\mathcal{D}^{2\alpha}_{\mathbf{X},\theta}([0,\hat{r}])$ is the mild solution of $(\ref{eq109})$-$(\ref{eq110})$  on $[0,\tilde{r}]$ with $y_{t}'=G(y_{t},\phi_{t-r})$.
			\end{Theorem}
			\begin{proof}
			We define the map $\mathcal{T}:  \mathcal{D}^{2\tilde{\alpha}}_{\mathbf{X},\theta}([0,\hat{r}])\rightarrow \mathcal{D}^{2\tilde{\alpha}}_{\mathbf{X},\theta}([0,\hat{r}])$ by 
				\begin{align*}
					\mathcal{T}(y,y')_{t}=( S_{t,0}y_{0}+\int^{t}_{0}S_{t,s}F(y_{s},\phi_{s-r})\textup{d}{s}+\int^{t}_{0}S_{t,s}G(y_{s},\phi_{s-r})\cdot\textup{d}\bar{\mathbf{X}}_{s}, G(y_{t},\phi_{t-r})).
				\end{align*}
				Due to $(ii)$ of Lemma \ref{lem006} and Lemma  \ref{lem007}, the map is well-defined. 
				
				For $\tilde{r}\in(0,\hat{r}]$, define a ball 
				\begin{align*}
					B(\tilde{r},y^{\star})&:=\big\{(y,y')\in  \mathcal{D}^{2\tilde{\alpha}}_{\mathbf{X},\theta}([0,\tilde{r}]):  
					(y_{0},y'_{0})= (y^{\star}, G(y^{\star},\phi_{-r}))\\
					&~~~~~~~~\text{and}~  \|y,y'\|_{\mathbf{X},2\tilde{\alpha},\theta,[0,\tilde{r}]}\leq (C_{1}+C_{3})(\|y^{\star}\|_{\theta}+\|\phi_{-r}\|_{\theta})+1\big\},
				\end{align*}
				where $C_{1}$ and $C_{3}$ are given in (\ref{eq117}) and (\ref{eq377}) respectively.
				
				We are going to show that there exists $\tilde{r}\in(0,\hat{r}]$ such that $\mathcal{T}$ is a self and contraction map from $B(\tilde{r},y^{\star})$ to $B(\tilde{r},y^{\star})$.
				
				Let us prove $\mathcal{T}$ is a self map firstly. 
				
				According to Lemmas \ref{lem007} and \ref{lem008}, we can  obtain 
				\begin{align*}
					&~~~~\|\mathcal{T}(y,y')\|_{\mathbf{X},2\tilde{\alpha},\theta,[0,\tilde{r}]}\\
					&\leq C_{1} \|y^{\star}\|_{\theta}+  C_{2}(1+\|y\|_{\infty,\theta, [0,\tilde{r}]}+\|\phi\|_{\infty,\theta,[-r,0]})\tilde{r}^{\lambda_{1}}+C_{3} (\|y^{\star}\|_{\theta}+\|\phi_{-r}\|_{\theta})\\
					&~~~~+C_{3} (1+\|y,y'\|_{\mathbf{X},2\tilde{\alpha},\theta,[0,\tilde{r}]}+
					\|\phi,\phi'\|_{\mathbf{X},2\tilde{\alpha},\theta,[-r,0]})^{2}\tilde{r}^{\lambda_{2}}\\
					&\leq (C_{1}+C_{3}) (\|y^{\star}\|_{\theta}+\|\phi_{-r}\|_{\theta})+C_{4}(1+\|y,y'\|_{\mathbf{X},2\tilde{\alpha},\theta,[0,\tilde{r}]}^{2}+
					\|\phi,\phi'\|_{\mathbf{X},2\tilde{\alpha},\theta,[-r,0]}^2)\tilde{r}^{\lambda_{1}\wedge\lambda_{2}},
				\end{align*}
				where $C_{4}$ is a constant dependent of $C_{2}$ and $C_{3}$.
				
				Then, we can obtain  a small enough $\tilde{r}_{1}$ such that 
				\begin{align*}
					\|\mathcal{T}(y,y')\|_{\mathbf{X},2\tilde{\alpha},\theta,[0,\tilde{r}]}\leq (C_{1}+C_{3}) (\|y^{\star}\|_{\theta}+\|\phi_{-r}\|_{\theta})+1.
				\end{align*}  
				
				Next, we show the contractivity of the map $\mathcal{T}$.
				
				Let $(x,x')\in  B(\tilde{r},y^{\star})$,  and set 
				$\tilde{F}_{t}= F(y_{t},\phi_{t-r})-F(x_{t},\phi_{t-r})$,  $\tilde{G}_{t}= G(y_{t},\phi_{t-r})-G(x_{t},\phi_{t-r})$. Then, 
				we have  
				\begin{align*}
					\mathcal{T}(y,y')_{t}- \mathcal{T}(x,x')_{t}=( \int^{t}_{0}S_{t,s}\tilde{F}_{s}\textup{d}{s}+\int^{t}_{0}S_{t,s}\tilde{G}_{s}\cdot\textup{d}\bar{\mathbf{X}}_{s}, \tilde{G}_{t}).
				\end{align*}
				
				By similar proof of Lemma \ref{lem007}, we can obtain 
				\begin{align*}
					\|(\int^{.}_{0}S_{.,s}\tilde{F}_{s}\textup{d}{s},0)\|_{\mathbf{X},2\tilde{\alpha}, \theta, [0,\tilde{r}]}\lesssim (\|y-x\|_{\infty,\theta, [0,\tilde{r}]})\tilde{r}^{\lambda_{1}}.
				\end{align*} 
				And, (ii) of Lemma \ref{lem006} yields
				\begin{align*}
					&~~~~\|(\int^{.}_{0}S_{.,s}\tilde{G}_{s}\cdot\textup{d}{\bar{\mathbf{X}}}_{s},\tilde{G}_{\cdot})\|_{\mathbf{X},2\tilde{\alpha}, \theta,[0,\tilde{r}]}\\
					&\lesssim \|G(y_{.},\phi_{.-r})-G(x_{.},\phi_{.-r}) ,    G^{\prime}(y_{.},\phi_{.-r})-G^{\prime}(x_{.},\phi_{.-r}), \bar{G}^{\prime}(y_{.},\phi_{.-r})-\bar{G}^{\prime}(x_{.},\phi_{.-r})\|_{ \bar{\mathbf{X}},2\tilde{\alpha},\theta-\sigma_{2},[0,\tilde{r}]}\tilde{r}^{\lambda_{2}}\\
					&\lesssim(\|y-x,y'-x'\|_{\mathbf{X},2\tilde{\alpha},\theta, [0,\tilde{r}]} )(1+\|y,y'\|_{\mathbf{X},2\tilde{\alpha},\theta,[0,\tilde{r}]}+
					\|x,x'\|_{\mathbf{X},2\tilde{\alpha},\theta,[0,\tilde{r}]}\|\phi,\phi'\|_{\mathbf{X},2\tilde{\alpha},\theta,[-r,0]})^{2}\tilde{r}^{\lambda_{2}}.
				\end{align*}
				Combining the above estimates, we know there exists $\tilde{r}_{2}$ such that 
				\begin{align*}
					\|\mathcal{T}(y,y')- \mathcal{T}(x,x')\|_{\mathbf{X},2\tilde{\alpha},\theta,[0,\tilde{r}]}< \|y-x,y'-x'\|_{\mathbf{X},2\tilde{\alpha},\theta,[0,\tilde{r}]} .
				\end{align*}
				
				Then, choosing $\tilde{r}=\tilde{r}_{1}\wedge\tilde{r}_{2}$, we obtain  $\mathcal{T}$ is a self and contraction map as desired.  Thus, we obtain that there exists  $(y,y')\in\mathcal{D}^{2\tilde{\alpha}}_{\mathbf{X},\theta}([0,\tilde{r}])$ as the solution of (\ref{eq109})-(\ref{eq110}) and $y'_{t}=G(y,\phi_{t-r})$ with the help of Banach fixed point theorem. 
				We remark that $\tilde{r}$ mainly depends on the initial value $y_{0}$  since other arguments ($\phi$,   $F$, $G$, $\lambda_{1}$ and $\lambda_{2}$) are fixed. 
				
				Finally,  we use Lemma \ref{lem011} to obtain $(y,y')\in\mathcal{D}^{2\alpha}_{\mathbf{X},\theta}([0,\tilde{r}])$
			\end{proof}

			{\bf \normalsize Step \uppercase\expandafter{\romannumeral 2}: A priori estimate
			}  
			
			Our final goal is to  show the solution of $(\ref{eq109})$-$(\ref{eq110})$ exits on $[0,r]$. 
			As preparations, we provide a priori estimate of $(y,y')$ as follows.
			
			In Lemma \ref{lem008}, we we are unaware of the relationship $y'_{t}=G(y_{t},\phi_{t-r})$, which leads to the appearance of $\|y,y'\|^{2}_{\mathbf{X},2\alpha,\theta,[0,\tilde{r}]}$. By Theorem \ref{the001}, we obtain $y'_{t}=G(y_{t},\phi_{t-r})$, so we can modify the proof of Lemma \ref{lem008} to arrive at the next result.

			\begin{Lemma}\label{lem101}
				If $(y,y')\in\mathcal{D}^{2\alpha}_{\mathbf{X},\theta}([0,\tilde{r}]) $ with $y'_{t}=G(y_{t},\phi_{t-r})$, and $\tilde{r}\in(0,\hat{r}]$ , then we have 
				\begin{align*}
					(\int^{.}_{0}S_{\cdot,s}G(y_{s},\phi_{s-r})\cdot\textup{d}{\bar{\mathbf{X}}}_{s},G(y_{.},\phi_{.-r}))\in \mathcal{D}^{2\alpha}_{\mathbf{X},\theta}([0,\tilde{r}]),
				\end{align*}
				and
				\begin{align*}
					&~~~~\|(\int^{.}_{0}S_{.,s}G(y_{s},\phi_{s-r})\cdot\textup{d}{\bar{\mathbf{X}}}_{s},G(y_{.},\phi_{.-r}))\|_{\mathbf{X},2\alpha, \theta,[0,\tilde{r}]}\\
					&\leq C_{5} [(1+\|y,y'\|_{\mathbf{X},2\alpha,\theta,[0,\tilde{r}]}	\|\phi,\phi'\|_{\mathbf{X},2\alpha,\theta,[-r,0]}+
					\|\phi,\phi'\|^{2}_{\mathbf{X},2\alpha,\theta,[-r,0]})\tilde{r}^{\alpha-\sigma_{2}}+\|y_{0}\|_{\theta}+ \|\phi'_{-r}\|_{\theta-\alpha}],
				\end{align*}
				where $C_{5}$ depends on   $ \bar{\mathbf{X}}$ and $G$.
			\end{Lemma}
			\begin{proof}
				By Corollary \ref{cor001}, we know 
				\begin{align*}
					&~~~\|G(y_{.},\phi_{.-r}),    G^{\prime}(y_{.},\phi_{.-r}), \bar{G}^{\prime}(y_{.},\phi_{.-r})\|_{ \bar{\mathbf{X}},2\alpha,\theta-\sigma_{2},[0,\tilde{r}]}\\
					&\lesssim 
					1+\|y,y'\|_{\mathbf{X},2\alpha,\theta,[0,\tilde{r}]}	\|\phi,\phi'\|_{\mathbf{X},2\alpha,\theta,[-r,0]}+
					\|\phi,\phi'\|^{2}_{\mathbf{X},2\alpha,\theta,[-r,0]}.
				\end{align*}
				Then,  we complete the proof due to $(iii)$ of Lemma \ref{lem006}.
			\end{proof}
			
			By Lemmas \ref{lem007} and \ref{lem101}, it is straightforward to obtain the next lemma. 
			\begin{Lemma}\label{lem102}
				Under the same settings of Lemma $\ref{lem101}$, we further assume that $(y,y')$ solves  $(\ref{eq109})$-$(\ref{eq110})$ on $[0,\tilde{r}]$. Then, we have 
				\begin{align}\label{eq509}
					\|y,y'\|_{\mathbf{X},2\alpha, \theta, [0,\tilde{r}]}\leq C_{6} (1+\|y_{0}\|_{\theta}+ 
					\tilde{r}^{\lambda_{3}}\|y,y'\|_{\mathbf{X},2\alpha, \theta, [0,\tilde{r}]}),
				\end{align}
				where $\lambda_{3}=(\alpha-\sigma_{2})\wedge(1-2\alpha)\wedge(1-\sigma_{1})$, and 
				$C_{6}$ depends on $ \bar{\mathbf{X}}$, $F$, $G$ and $\phi$.
			\end{Lemma}
			Now, we can show a priori estimate of $(y,y')$.
			\begin{Lemma}\label{lem103}
				Under the same settings of Lemma $\ref{lem102}$, we have 
				\begin{align*}
					\|y\|_{\infty, \theta, [0,\tilde{r}]}\leq C_{7} u e^{C_{8}\tilde{r}},
				\end{align*}
				where   $u=1\vee\|y_{0}\|_{\theta}$,
				$C_{7}$ and $C_{8}$ depend on $ \bar{\mathbf{X}}$, $F$, $G$ and $\phi$.
			\end{Lemma}
			\begin{proof}
				The proof is similar to that of Lemma 3.6 of \cite{MR4431448}. Here, we just give a sketch. 
				
				Observing (\ref{eq509}), we can choose large enough integer $M$ such that $C_{6}(\frac{\tilde{r}}{M})^{\lambda_{3}}\in(\frac{1}{4},\frac{1}{2})$. Then, we have  on $[0,\frac{\tilde{r}}{M}]$:
				\begin{align*}
					\|y\|_{\infty,\theta, [0,\frac{\tilde{r}}{M}] }\leq \|y,y'\|_{\mathbf{X},2\alpha, \theta, [0,\frac{\tilde{r}}{M}]}\leq 4C_{6}u.
				\end{align*}
				By concatenation argument, for $k=1,\cdots,M-1$, we further derive  on $[\frac{k\tilde{r}}{M}, \frac{(k+1)\tilde{r}}{M}]$:
				\begin{align*}
					\|y\|_{\infty,\theta, [\frac{k\tilde{r}}{M}, \frac{(k+1)\tilde{r}}{M}]}\leq \|y,y'\|_{\mathbf{X},2\alpha, \theta, [\frac{k\tilde{r}}{M},\frac{(k+1)\tilde{r}}{M}]}\leq (4C_{6})^{k+1}u.
				\end{align*}
				In fact, the assertion is reasonable since $\|\phi_{.-r},\phi'_{.-r}\|_{\mathbf{X},2\alpha,\theta}$ is fixed on $[0,r]$, which means that $C_{6}$ do not change in each iteration. Therefore, we obtain on $[0, \tilde{r}]$
				\begin{align*}
					\|y\|_{\infty,\theta, [0, \tilde{r}]}\leq (4C_{6})^{M}u.
				\end{align*}
				Noting $N<(4C_{6})^{\frac{1}{\lambda_{3}}}\tilde{r}$,  we can choose $C_{7}=(4C_{6})^{(4C_{6})^{\frac{1}{\lambda_{3}}}}$ and $C_{8}= \log(4C_{6})$. 
			\end{proof}
			
			{\bf \normalsize Step \uppercase\expandafter{\romannumeral 3}: Existence of global mild solution
			}

			Thanks to Theorem \ref{the001} and Lemma \ref{lem103}, we can ensure that the mild solution of   (\ref{eq109})-(\ref{eq110}) exists on $[0,r]$.
			\begin{Theorem}\label{the002}
				There exists  $ (y,y')\in\mathcal{D}^{2\tilde{\alpha}}_{\mathbf{X},\theta}([0,r])\cap\mathcal{D}^{2\alpha}_{\mathbf{X},\theta}([0,r])$ such that $(y,y')$ is the mild solution of $(\ref{eq109})$-$(\ref{eq110})$  on $[0,r]$ with $y_{t}'=G(y_{t},\phi_{t-r})$.
			\end{Theorem}
			\begin{proof}
				Take $u=1\vee \|y_{0}\|_{\theta}$, and set $\tilde{u}= C_{7}ue^{C_{8}T}$. As we do in Theorem \ref{the001}, we choose enough large integer $M$ such that $(y,y')$ solves (\ref{eq109})-(\ref{eq110})
				on $[0,\frac{r}{M}]$ according to $\|y_{0}\|_{\theta}\leq \tilde{u}$. Then, by Lemma \ref{lem103}, it holds that 
				\begin{align*}
					\|y_{\frac{r}{M}}\|_{\theta}\leq	\sup_{t\in[0,\frac{r}{M}]}\|y_{t}\|_{\theta}\leq  C_{7} u e^{C_{8}\frac{r}{M}}<\tilde{u}.
				\end{align*}
				Thus, we can use Theorem \ref{the001} again to obtain the existence of the mild solution of  (\ref{eq109})-(\ref{eq110}) with  initial value $y_{\frac{r}{M}}$ on  $[\frac{r}{M},\frac{2r}{M}]$. By repeating this  procedure, we can extend the interval of the solution from $[0,\frac{r}{M}]$ to  $[0,r]$.
			\end{proof}

			\begin{Theorem}\label{Cor002}
				There exists  $ (y,y')\in\mathcal{D}^{2\tilde{\alpha}}_{\mathbf{X},\theta}([0,T])\cap\mathcal{D}^{2\alpha}_{\mathbf{X},\theta}([0,T])$ such that $(y,y')$ is the mild solution of $(\ref{eq109})$-$(\ref{eq110})$  on $[0,T]$ with $y_{t}'=G(y_{t},y_{t-r})$. 
			\end{Theorem}
			\begin{proof}
				By Theorem \ref{the002}, we know that $(y,y')$ is the mild solution of (\ref{eq109})-(\ref{eq110}). Then, we can consider (\ref{eq109})$-$(\ref{eq110}) as  follows:
					\begin{align}
					y_{t}&=S_{t,0}y_{0}+\int^{t}_{0}S_{t,s}F(y_{s},y_{s-r})\textup{d}{s}+\int^{t}_{0}S_{t,s}G(y_{s},y_{s-r})\cdot\textup{d}\bar{\mathbf{X}}_{s},\label{eq1091}\\
					y_{t}&=y_{t}, ~~ t\in[0,r].\label{eq1101}
				\end{align}
We note that 	$(y,y')\in 
\mathcal{D}^{2\tilde{\alpha}}_{\mathbf{X},\theta}([0,r])\cap\mathcal{D}^{2\alpha}_{\mathbf{X},\theta}([0,r])$. Thus,  we can use 		Theorem \ref{the002} to prove that there exists the mild solution of (\ref{eq1091})-(\ref{eq1101}) on $[r,2r]$. In other  words,  (\ref{eq109})-(\ref{eq110}) has the mild solution on  $[0,2r]$. By further iteration, the global mild solution on $[0,T]$ can be obtained.
			\end{proof}

			\subsection{Stability of solutions}
			
			In the subsection, we hope to study the distance between each solution of two delay rough PDEs.

			Let $\mathbf{Y}=(Y, \mathbb{Y})\in\mathcal{C}^{\alpha}([0,T], \mathbb{R}^d)$, $\bar{\mathbf{Y}}=(Y, \mathbb{Y},\mathbb{Y}(-r))\in\bar{\mathcal{C}}^{\alpha}([0,T], \mathbb{R}^d)$ with  $T>0$, $r>0$, $\alpha\in(\frac{1}{3},\frac{1}{2}]$. We assume  that
			$(\psi,\psi')$ is a controlled path based on $\mathbf{Y}$ belonging to $\mathcal{D}^{2\tilde{\alpha}}_{\mathbf{Y},\theta}([-r,0])\cap\mathcal{D}^{2\alpha}_{\mathbf{Y},\theta}([-r,0])$ with $\theta\in\mathbb{R}$. Considering
			\begin{align}
				\textup{d}y_{t}&=[A y_{t}+F(y_{t}, y_{t-r})]\textup{d}t+G(y_{t},y_{t-r})\cdot\textup{d}\bar{\mathbf{X}}_{t},\label{eq201}\\
				y_{t}&=\phi_{t}, ~~ t\in[-r,0],\label{eq202}
			\end{align}
			and 
			\begin{align}
				\textup{d}z_{t}&=[A z_{t}+F(z_{t}, z_{t-r})]\textup{d}t+G(z_{t},z_{t-r})\cdot\textup{d}\bar{\mathbf{Y}}_{t},\label{eq203}\\
				z_{t}&=\psi_{t}, ~~ t\in[-r,0],\label{eq205}
			\end{align}
			we would like to estimate the distance $\rho_{2\hat{\alpha},2\alpha, \theta}(y,z)$ between $(y,y')$ and $(z,z')$, , where $\hat{\alpha}\in(\sigma_{2},\alpha)$ satisfies $3\hat{\alpha}-2\alpha-\sigma_{2}>0$. 
			
			Theorem \ref{Cor002} allows us to know there exists  $(y,y')\in\mathcal{D}^{2\alpha}_{\mathbf{X},\theta}([0,T])$  and $(z,z')\in\mathcal{D}^{2\alpha}_{\mathbf{X},\theta}([0,T])$ as the mild solutions of $(\ref{eq201})$-$(\ref{eq202})$ and $(\ref{eq203})$-$(\ref{eq205})$, respectively. 
			We set 
			\begin{align*}
				\mathcal{M}&:=\|y,y'\|_{\mathbf{X},2\alpha, \theta, [0,T] }+\|\phi,\phi'\|_{\mathbf{X},2\alpha, \theta, [-r,0]}+\|z,z'\|_{\mathbf{Y},2\alpha, \theta, [0,T]}+\|\psi,\psi'\|_{\mathbf{Y},2\alpha, \theta, [-r,0]}\\
				&~~~+ 
				\rho_{\alpha, [0,T]}(\bar{\mathbf{X}})
				+\rho_{\alpha, [0,T]}(\bar{\mathbf{Y}}),
			\end{align*}
			and $\mathcal{U}:=\rho_{2\hat{\alpha},2\alpha, \theta, [-r,0]}(\phi,\psi)+\rho_{\alpha, [0,T]}(\bar{\mathbf{X}},\bar{\mathbf{Y}})$.

			
		
		\begin{Theorem}
			Let $(y,y')\in\mathcal{D}^{2\alpha}_{\mathbf{X},\theta}([0,T])$  and $(z,z')\in\mathcal{D}^{2\alpha}_{\mathbf{X},\theta}([0,T])$ be the mild solutions of $(\ref{eq201})$-$(\ref{eq202})$ and $(\ref{eq203})$-$(\ref{eq205})$, respectively. Then, we have 
			\begin{align*}
				\rho_{2\hat{\alpha},2\alpha, \theta, [0,T]}(y,z)\lesssim_{\mathcal{M}} \mathcal{U}.
			\end{align*}
		\end{Theorem}
		\begin{proof}
			In order to complete the proof, we need to employ lots of similar arguments given in previous lemmas. Here, we only provide some key steps for simplicity.
			
			We  consider $\rho_{2\hat{\alpha},2\alpha, \theta, [0,\tilde{r}]}(y,z)$ on the interval $[0,\tilde{r} ]$ with $\tilde{r}\in(0,\hat{r}]$. Set $\nu:=\min\{1-\sigma_{1}, 1-2\hat{\alpha}, \alpha-\hat{\alpha},\frac{\hat{\alpha}(\alpha-\sigma_{2})}{\alpha}, 3\hat{\alpha}-2\alpha-\sigma\}$.
			
			Similar to the proof of Lemma \ref{lem007}, it is easy to obtain 
			\begin{align*}
				&~~~\rho_{2\hat{\alpha},2\alpha, \theta, [0,\tilde{r}]}(\int^{\cdot}_{0}S_{\cdot,u}F(y_{u},\phi_{u-r})\textup{d}u,\int^{\cdot}_{0}S_{\cdot,u}F(z_{u},\psi_{u-r})\textup{d}u)\\
				&\lesssim_{\mathcal{M}} \rho_{2\hat{\alpha},2\alpha, \theta, [0,\tilde{r}]}(y,z)   \tilde{r}^{\nu}+\mathcal{U}.
			\end{align*}
			And, analogous arguments in Lemma \ref{lem777} yield 
			\begin{align*}
				&~~~\rho_{2\hat{\alpha},2\alpha,  \theta, [0,\tilde{r}]}(\int^{\cdot}_{0}S_{\cdot,s}G(y_{s},\phi_{s-r})\cdot\textup{d}\bar{\mathbf{X}}_{s},\int^{\cdot}_{0}S_{\cdot,s}G(z_{s},\psi_{s-r})\cdot\textup{d}\bar{\mathbf{Y}}_{s})\\
				&\lesssim_{\mathcal{M}} \rho_{2\hat{\alpha},2\alpha, \theta, [0,\tilde{r}]}(y,z)   \tilde{r}^{\nu}+\mathcal{U}.
			\end{align*}

		Consequently, we deduce 
		\begin{align*}
			\rho_{2\hat{\alpha},2\alpha, \theta, [0,\tilde{r}]}(y,z) \lesssim_{\mathcal{M}}\rho_{2\hat{\alpha},2\alpha, \theta, [0,\tilde{r}]}(y,z) \tilde{r}^{\nu}+\mathcal{U}+ \|y_{0}-z_{0}\|_{\theta}.
		\end{align*}
		Then, choosing small enough $\tilde{r}$, we derive on $[0,\tilde{r}]$:
		\begin{align*}
			\rho_{2\hat{\alpha},2\alpha, \theta, [0,\tilde{r}]}(y,z) \lesssim_{\mathcal{M}}\mathcal{U}+\|y_{0}-z_{0}\|_{\theta}.
		\end{align*}
		By iteration, the desired result can be obtained.
	\end{proof}
	
	\section{Convergence with respect to the delay}
	In this section, we aim to investigate whether the solution of a delay equation can converge to the solution of a corresponding equation without delay as the time delay approaches zero.
	
	We introduce some notations and assumptions which are used in the section.
	
	$\hat{\mathbf{H}} \mathbf{1}$:  Fix $T>0$, $r_{0}>0$, $\alpha\in(\frac{1}{3},\frac{1}{2}]$, $ \sigma_{2}\in (\frac{\alpha}{2}, \alpha)$ and $\tilde{\alpha}\in (\sigma, \alpha)$. 
	Fix   $\bar{\alpha}\in (\frac{4\alpha}{5},\alpha )$  satisfying $ \sigma_{2}+2\bar{\alpha}-2\alpha\geq 0$ and  $3\bar{\alpha}-2\alpha-\sigma_{2}\geq 0$.	 For any $r\in[0,r_{0}]$, Let $\mathbf{X}=(X, \mathbb{X})\in\mathcal{C}^{\alpha}([0,T], \mathbb{R}^d)$, $\bar{\mathbf{X}}=(X, \mathbb{X},\mathbb{X}(-r))\in\mathcal{C}^{\alpha}([0,T], \mathbb{R}^d)$. We assume
	\begin{align}\label{eq707}
		\sup_{t,s\in[r,T]}\frac{\|\mathbb{X}_{t,s}-\mathbb{X}(-r)_{t,s}\|_{\mathbb{R}^{d}}}
		{(t-s)^{2\bar{\alpha}}}\leq h(r),
	\end{align}
	where $h(r)=o(r)$.

	$\hat{\mathbf{H}}\mathbf{2}$:  $(\phi,\phi')$ is a controlled path based on $\mathbf{X}$ belonging to $\mathcal{D}^{2\tilde{\alpha}}_{\mathbf{X},\theta}([-r_{0},0])\cap \mathcal{D}^{2\alpha}_{\mathbf{X},\theta}([-r_{0},0])$ with $\theta\in\mathbb{R}$.
	
	$\hat{\mathbf{H}} \mathbf{3}$:  Nonlinear operator $\mathcal{F}\in  \text{Lip}_{\theta,-\sigma_{1}}(\mathcal{B}, \mathcal{B})$ with $\sigma_{1}\in(0,1)$.
	
	$\hat{\mathbf{H}} \mathbf{4}$:  Nonlinear operator $\mathcal{G}=(G_{1},\cdots,G_{d})$, such that $\mathcal{G}_{i}\in  \mathcal{C}^{3}_{\theta-2\alpha,-\sigma_{2}}(\mathcal{B}, \mathcal{B})$  for $i=1,\cdots,d$. Moreover, we assume that for $1\leq i,j\leq d$, fixed $x_{1},x_{2},y_{1},y_{2}\in \mathcal{B}_{\theta-\alpha}$,   it holds
	\begin{align}\label{eq223}
		\|D\mathcal{G}_{i}(x_{1})\circ\mathcal{G}_{j}(y_{1})-D\mathcal{G}_{i}(x_{2})\circ\mathcal{G}_{j}(y_{2})\|_{ \theta-2\alpha-\sigma_{2}    }\lesssim \|x_{1}-x_{2}\|_{\theta-\alpha}+\|y_{1}-y_{2}\|_{\theta-\alpha}.
	\end{align}
	
	\begin{Remark} ~~
		
		(i)  In $\hat{\mathbf{H}} 1$,   $\bar{\alpha}$ is chosen for the convenience of subsequent estimations.
		
		(ii)	  
		According to $\hat{\mathbf{H}} 1$,	we note that $\rho_{\alpha, [0,T]}(\bar{\mathbf{X}})$ is uniformly bounded  with respect to $r\in[0,r_{0}]$.  Moreover, we will state Le\'vy area generated by Brownian motion can satisfies $(\ref{eq707})$ later.
		
		(iii) The condition  $(\ref{eq223})$ holds if $\mathcal{G}$ is either linear or uniformly bounded.
	\end{Remark}

	
	In this section,  we consider
	\begin{align}
		y_{t}&=[A y_{t}+\mathcal{F}(y_{t})]\textup{d}t+\mathcal{G}(y_{t-r})\cdot\textup{d}\bar{\mathbf{X}}_{t},\label{eq701}\\
		y_{t}&=\phi_{t}, ~~ t\in[-r,0],\label{eq702}
	\end{align}
	and 
	\begin{align}
		z_{t}&=[A z_{t}+\mathcal{F}(z_{t})]\textup{d}t+\mathcal{G}(z_{t})\cdot\textup{d}\mathbf{X}_{t},\label{eq703}\\
		z_{0}&=\phi_{0}.\label{eq705}
	\end{align}
	Comparing two systems,  we  observe that (\ref{eq701})-(\ref{eq702}) can be regarded as the (\ref{eq703})-(\ref{eq705}) perturbed by delay.  From Theorem \ref{the002}, it follows that 
	there exists the mild solution $(y,y') \in\mathcal{D}^{2\alpha}_{\mathbf{X},\theta}([0,T]) \cap \mathcal{D}^{2\tilde{\alpha}}_{\mathbf{X},\theta}([0,T]) $  of (\ref{eq701})-(\ref{eq702}), and the mild solution $(z,z')\in\mathcal{D}^{2\alpha}_{\mathbf{X},\theta}([0,T])$ of (\ref{eq703})-(\ref{eq705}).
	Our aim is to estimate $\rho_{2\bar{\alpha},2\alpha, \theta-\alpha,[0,T]}(y,z)$ (i.e., the distance   between  $(y,y')$  and
	$(z,z')$).

\begin{Remark} 
In the process of estimating the distance, we need to repeatedly utilize the H\"{o}lder continuous property of $y$ (i.e., $\|\delta y\|_{\alpha,\theta-\alpha}<\infty$). Therefore,  the distance is considered with less regularity of spatial due to this technical reason.
\end{Remark}

For simplicity, set
\begin{align*}
\hat{\mathcal{M}}:= \|y,y'\|_{\mathbf{X},2\alpha, \theta,[0,T]}+\|\phi,\phi'\|_{\mathbf{X},2\alpha, \theta,[-r_{0},0]}+\|z,z'\|_{\mathbf{X},2\alpha, \theta,[0,T]}
+\rho_{\alpha, [0,T]}(\bar{\mathbf{X}}).
\end{align*}

We are going to take account into the distance on $[0,r]$ and $[r,T]$, respectively.
Before studying the distance, an important task  is to  emphasize  $\|y,y'\|_{\mathbf{X},2\alpha, \theta,[0,T]}$ independent of $r$.

\begin{Lemma}
Let $(y,y')\in\mathcal{D}^{2\alpha}_{\mathbf{X},\theta}([0,T])$ be the mild solution of $(\ref{eq701})$-$(\ref{eq702})$. Then, we have $\|y,y'\|_{\mathbf{X},2\alpha, \theta,[0,T]}$ is independent of $r$.
\end{Lemma}
\begin{proof}
Similar to Lemma \ref{lem102},  we know for any $\tilde{r}\in[0, \hat{r}]$:
\begin{align}
	\|y,y'\|_{\mathbf{X},2\alpha, \theta, [0,\tilde{r}]  }&\leq \hat{C}_{6} (1+\|\phi,\phi'\|^{2}_{\mathbf{X},2\alpha,\theta, [-r_{0},0]}+
	\tilde{r}^{\hat{\lambda}_{3}}\|y,y'\|_{\mathbf{X},2\alpha, \theta, [0,\tilde{r}] }),\label{eq219}\\
	\|y,y'\|_{\mathbf{X},2\alpha, \theta, [r,r+\tilde{r}]  }&\leq \hat{C}_{6} (1+
	\|y,y'\|_{\mathbf{X},2\alpha, \theta, [-r,r] }+
	\tilde{r}^{\hat{\lambda}_{3}}\|y,y'\|_{\mathbf{X},2\alpha, \theta, [r,r+\tilde{r}] }),\label{eq220}
\end{align}
where $\hat{\lambda}_{3}=(1-2\alpha)\wedge(1-\sigma_{1})\wedge(\alpha-\sigma_{2})$, and 
$\hat{C}_{6}$ depends on  $ \bar{\mathbf{X}}$, $\mathcal{F}$, $\mathcal{G}$. 
Choose large enough $M$ such that $\hat{C}_{6} (\frac{T}{M})^{\hat{\lambda}_{3}}<\frac{1}{2}$ and $\frac{T}{M}\leq 1\wedge r_{0}$. 

On the one hand,	if $\frac{kT}{M}=r$ with positive integer $k$, then (\ref{eq219})
yields that			\begin{align*}
	\|y,y'\|_{\mathbf{X},2\alpha, \theta, [\frac{(\hat{k}-1)T}{M},\frac{\hat{k}T}{M}]  }\leq  4\hat{C}_{6}l,~ 1\leq \hat{k}\leq k,
\end{align*}
where $l:=1\vee  \|\phi,\phi'\|^{2}_{\mathbf{X},2\alpha,\theta, [-r_{0},0]}$. 
Collecting each interval, it follows 
\begin{align*}
	\|y,y'\|_{\mathbf{X},2\alpha, \theta, [0,r]  }\leq 4k\hat{C}_{6}l.
\end{align*}
By iteration, we can obtain
\begin{align*}
	\|y,y'\|_{\mathbf{X},2\alpha, \theta, [0,T]  }\leq C(M,k,l, \hat{C}_{6}),
\end{align*}
where $M$ and $k$ only depend on $\hat{C}_{6}$, $T$, $\hat{\lambda}_{3}$ and $r_{0}$

On the other hand,	if $r<\frac{T}{M}\leq \min\{1, r_{0}\}$, then
we  also  derive 
\begin{align*}
	\|y,y'\|_{\mathbf{X},2\alpha, \theta, [0,r]  }\leq  4\hat{C}_{6}l.
\end{align*}
Furthermore, it holds
\begin{align}\label{eq221}
	\|y,y'\|_{\mathbf{X},2\alpha, \theta, [-r,r]  }\leq  8\hat{C}_{6}l.
\end{align}
Thus, (\ref{eq220}) and (\ref{eq221}) imply  
\begin{align*}
	\|y,y'\|_{\mathbf{X},2\alpha, \theta, [r,\frac{T}{M}]}\leq  32\hat{C}_{6}l,
\end{align*}
from which we can deduce 
\begin{align*}
	\|y,y'\|_{\mathbf{X},2\alpha, \theta, [-r,\frac{T}{M}]}\leq  64\hat{C}_{6}l.
\end{align*}
We also notice 
\begin{align*}
	\|y,y'\|_{\mathbf{X},2\alpha, \theta, [\frac{T}{M},\frac{2T}{M}]  }&\leq \hat{C}_{6} (1+
	\|y,y'\|_{\mathbf{X},2\alpha, \theta, [-r,\frac{T}{M}] }+
	\tilde{r}^{\hat{\lambda}_{3}}\|y,y'\|_{\mathbf{X},2\alpha, \theta, [\frac{T}{M},\frac{2T}{M}]}).
\end{align*}
Then, it holds 
\begin{align*}
	\|y,y'\|_{\mathbf{X},2\alpha, \theta, [\frac{T}{M},\frac{2T}{M}]}\leq  C(\hat{C}_{6},l).
\end{align*}
By iterating this process and piecing together the estimations on each interval, the result  is also obtained as desired.
\end{proof}

The distance $\rho_{2\bar{\alpha},2\alpha, \theta-\alpha, [0,r]}(y,z)$ for $[0,r]$ follows easily.
\begin{Lemma}\label{lem707}
Let $(y,y') \in\mathcal{D}^{2\alpha}_{\mathbf{X},\theta}([0,T])$  and $(z,z')\in\mathcal{D}^{2\alpha}_{\mathbf{X},\theta}([0,T])$ be the mild solution of  $(\ref{eq701})$-$(\ref{eq702})$ and $(\ref{eq703})$-$(\ref{eq705})$, respectively.  Then, there exists some constant $c>0$ such that
\begin{align*}
\rho_{2\bar{\alpha},2\alpha, \theta-\alpha,[0,r]}(y,z)\lesssim_{\hat{\mathcal{M}}} r^{c}.
\end{align*}
\end{Lemma}
\begin{proof}
We estimate each observation in the distance as follows.

From $(y,y') \in\mathcal{D}^{2\alpha}_{\mathbf{X},\theta}([0,T])$  and $(z,z')\in\mathcal{D}^{2\alpha}_{\mathbf{X},\theta}([0,T])$, it holds 
\begin{align*}
&~~~~\|y-z\|_{\infty,\theta-\alpha,[0,r]}\\
&\lesssim \|\int^{\cdot}_{0}S_{\cdot,s}[\mathcal{F}(y_{s})-\mathcal{F}(z_{s})]\textup{d}{s}\|_{\infty,\theta-\alpha,[0,r]}+\|\int^{\cdot}_{0}S_{\cdot,s}[\mathcal{G}(\phi_{s-r})-\mathcal{G}(z_{s})]\cdot\textup{d}\bar{\mathbf{X}}_{s}\|_{\infty,\theta-\alpha,[0,r]}\\
&\lesssim r^{1-\sigma_{1}}(1+\|y\|_{\infty,\theta,[0,r]}+\|z\|_{\infty,\theta,[0,r]})+r^{\alpha-\sigma_{2}}\|\mathcal{G}(\phi_{\cdot-r})-\mathcal{G}(z)\|_{\bar{\mathbf{X}},2\alpha, \theta-\alpha-\sigma_{2},[0,r]}\\
&~~~+\|\mathcal{G}(\phi_{-r})-\mathcal{G}(z_{0})\|_{\theta-\alpha-\sigma_{2}}\\
&\lesssim r^{1-\sigma_{1}}+r^{ \alpha-\sigma_{2}}+r^{\alpha},
\end{align*}
where we use Lemma \ref{lem006} in the second inequality, and Lemma \ref{lem001} in the last one.

Note $y'_{t}=\mathcal{G}(\phi_{t-r})$ and $z'_{t}=\mathcal{G}(z_{t})$. Applying the assumptions on $\mathcal{G}$,  we derive for $t\in[0,r]$,
\begin{align*}
&~~~~\|\mathcal{G}(\phi_{t-r})-\mathcal{G}(z_{t})\|_{\theta-2\alpha}\\
&=\|\mathcal{G}(\phi_{t-r})-\mathcal{G}(\phi_{0})+\mathcal{G}(z_{0})  -\mathcal{G}(z_{t})\|_{\theta-2\alpha}\\
&\lesssim_{\hat{\mathcal{M}}} \|\phi_{t-r}-\phi_{0}\|_{ \theta-\alpha+\sigma_{2}}+\|z_{0}-z_{t}\|_{ \theta-\alpha+\sigma_{2}}\\
&\lesssim_{\hat{\mathcal{M}}} (\|\phi_{t-r}-\phi_{0}\|_{\theta}+\frac{\|\phi_{t-r}-\phi_{0}\|_{\theta-\alpha}}{(t-r)^{\alpha}})(t-r)^{\alpha-\sigma_{2}}+  (\|z_{t}-z_{0}\|_{\theta}+\frac{\|z_{t}-z_{0}\|_{\theta-\alpha}}{t^{\alpha}})t^{\alpha-\sigma_{2}}\\
&\lesssim_{\hat{\mathcal{M}}} r^{\alpha-\sigma_{2}},
\end{align*}  
where we use interpolation inequality (\ref{eq053}) in the third line.

And, it not hard to deduce 
\begin{align*}
&~~~~\|\mathcal{G}(\phi_{t-r})- \mathcal{G}(\phi_{s-r})-
\mathcal{G}(z_{t})+\mathcal{G}(z_{s})
\|_{\theta-3\alpha}\\
&\lesssim \|\phi_{t-r}-\phi_{s-r}\|_{\theta-2\alpha+\sigma_{2}}+\|z_{t}-z_{s}\|_{\theta-2\alpha+\sigma_{2}}.
\end{align*}  
which implies $ \|\delta(y'-z')\|_{ \bar{\alpha},\theta-3\alpha,[0,r] } \lesssim_{ \hat{\mathcal{M}}} r^{\alpha-\bar{\alpha}}$. 


Easily, for $i=1,2$, it follows
\begin{align*}
\|R^{y}-R^{z}\|_{i\bar{\alpha}, \theta-\alpha-i\alpha,[0,r]}
\lesssim (\|R^{y}\|_{i\alpha, \theta-i\alpha,[0,r]}+\|R^{z}\|_{i\alpha, \theta-i\alpha,[0,r]})r^{i\alpha-i\bar{\alpha}}.
\end{align*}

Collecting the above estimates, we obtain suitable $c$ as desired.
\end{proof}

The distance on $[r,T]$ will be  present  in the followings.
In order to achieve  our goal, we need to study the relationship between $y_{t}$ and $y_{t-r}$ as preparations.   

\begin{Lemma}\label{lem107}
Let $(y,y') \in\mathcal{D}^{2\alpha}_{\mathbf{X},\theta}([0,T])$    be the mild solution of  $(\ref{eq701})$-$(\ref{eq702})$. 
For $r\leq u<v\leq T $,  there exists some constant $c>0$ such that
\begin{align*}
\sup_{t\in[u,v]}	\|\int^{t}_{u}S_{t,s}[\mathcal{G}(y_{s})-\mathcal{G}(y_{s-r})]\cdot\textup{d}\bar{\mathbf{X}}_{s}\|_{\theta-\alpha}\lesssim_{ \hat{\mathcal{M}},T } r^{c}+h(r).
\end{align*}
\end{Lemma}
\begin{proof}
Notice $\mathcal{G}(y_{t})-\mathcal{G}(y_{t-r})$ is a delayed controlled path on $[r,T]$ due to Lemma \ref{lem001}.

Set $\mathbb{G}_{t}= \mathcal{G}(y_{t})-\mathcal{G}(y_{t-r})$,  and rewrite $\int^{t}_{s}S_{t,\tau}\mathbb{G}_{\tau}\cdot\textup{d}\bar{\mathbf{X}}_{\tau}=
I^{1}_{t,s}+I^{2}_{t,s}$,
where 
\begin{align*}
I^{1}_{t,s}&=\int^{t}_{s}S(t-u)\mathbb{G}_{s}\cdot\textup{d}\bar{\mathbf{X}}_{s}-S_{t,s}\big\{\mathbb{G}_{s}\cdot\delta X_{t,s}+\mathbb{G}'_{s}:\mathbb{X}_{t,s}+   \bar{\mathbb{G}}'_{s}:\mathbb{X}(-r)_{t,s}\big\},\\
I^{2}_{t,s}&=S_{t,s}\big\{\mathbb{G}_{s}\cdot\delta X_{t,s}+\mathbb{G}'_{s}:\mathbb{X}_{t,s}+\bar{\mathbb{G}}'_{s}:\mathbb{X}(-r)_{t,s}\big\}.
\end{align*}


We start with the estimate of $I^{1}_{t,s}$ by (ii) of Lemma \ref{lem003} . 
Indeed, we need to consider these terms
$\mathbb{G}_{s}\cdot\delta X_{t,s}$, $\mathbb{G}^{\prime}_{s}:\mathbb{X}_{t,s}+   \bar{\mathbb{G}}^{\prime}_{s}:\mathbb{X}(-r)_{t,s}$, 
$\delta X_{t,l}\cdot\bar{R}^{\mathbb{G}}_{l,s}$ as well as $\delta \mathbb{G}^{\prime}_{l,s}:\mathbb{X}_{t,l}+   \delta \bar{\mathbb{G}}^{\prime}_{l,s}:\mathbb{X}(-r)_{t,l}$
as follows.

For the first term, we deduce 
\begin{align*}
&~~~\sup_{s,t\in[u,v]}\frac{\|\mathbb{G}_{s}\cdot\delta X_{t,s}\|_{\theta-\sigma_{2}-\alpha}     }{(t-s)^{\bar{\alpha}}}\\
&\lesssim 	\sup_{ t\in[r,T]}\|  \mathbb{G}_{t}   \|_{\theta-\sigma_{2}-\alpha} T^{\alpha-\bar{\alpha}}\\
&\lesssim _{\hat{\mathcal{M}}}\sup_{ t\in[r,T]}\frac{\|  y_{t}-y_{t-r} \|_{\theta-\alpha}}{r^{\alpha}}r^{\alpha}T^{\alpha-\bar{\alpha}}\\
&\lesssim_{\hat{\mathcal{M}}} \sup_{ t,s\in[0,T]} \frac{\|  y_{t}-y_{s} \|_{\theta-\alpha}}{(t-s)^{\alpha}}r^{\alpha}T^{\alpha-\bar{\alpha}}\\
&\lesssim_{\hat{\mathcal{M}},T }r^{\alpha}.
\end{align*}

For the second term, we note 
\begin{align}
&~~~~\mathbb{G}^{\prime}_{s}:\mathbb{X}_{t,s}+   \bar{\mathbb{G}}^{\prime}_{s}:\mathbb{X}(-r)_{t,s}\label{eq711}\\
&=D\mathcal{G}(y_{s})\circ\mathcal{G}(y_{s-r}):\mathbb{X}_{t,s}-D\mathcal{G}(y_{s-r})\circ\mathcal{G}(y_{s-2r}):\mathbb{X}(-r)_{t,s}\notag\\
&=\sum^{3}_{j=1}I^{1,j}_{t,s},\notag
\end{align}
where 
\begin{align*}
I^{1,1}_{t,s}&=[D\mathcal{G}(y_{s})-D\mathcal{G}(y_{s-r})]\circ \mathcal{G}(y_{s-r}) :\mathbb{X}_{t,s}\\
I^{1,2}_{t,s}&=D\mathcal{G}(y_{s-r})\circ [\mathcal{G}(y_{s-r})-\mathcal{G}(y_{s-2r})] :\mathbb{X}_{t,s},\\
I^{1,3}_{t,s}&=D\mathcal{G}(y_{s-r})\circ \mathcal{G}(y_{s-2r}):(\mathbb{X}_{t,s}-\mathbb{X}(-r)_{t,s}).\\
\end{align*}
Then, we derive 
\begin{align*}
&~~~~\sup_{t,s\in[u,v]}\frac{\|I^{1,1}_{t,s}\|_{\theta-\sigma_{2}-\alpha-\bar{\alpha}}}{(t-s)^{2\bar{\alpha}}}\\
&\lesssim \sup_{s\in[r,T]} (\|y_{s}-y_{s-r}  \|_{\theta-\alpha-\bar{\alpha}}\|y_{s-r}\|_{\theta+\sigma_{2}-\alpha-\bar{\alpha}   })   T^{2\alpha-2\bar{\alpha}}\\
&\lesssim_{\hat{\mathcal{M}},T} r^{\alpha}.
\end{align*}
As for $I^{1,2}_{t,s}$, we have 
\begin{align*}
&~~~~\sup_{t,s\in[u,v]}\frac{\|I^{1,2}_{t,s}\|_{\theta-\sigma_{2}-\alpha-\bar{\alpha}}}{(t-s)^{2\bar{\alpha}}}\\
&\lesssim  \sup_{t\in[r,T]} (\|y_{t-r}-y_{t-2r}\|_{\theta+\sigma_{2}-\alpha-\bar{\alpha}   })   T^{2\alpha-2\bar{\alpha}}\\
&\lesssim_{\hat{\mathcal{M}},T} r^{\alpha}.
\end{align*}
And, based on the assumption (\ref{eq707}), it is easy to have 
\begin{align*}
&~~~~\sup_{t,s\in[u,v]}\frac{\|I^{1,3}_{t,s}\|_{\theta-\sigma_{2}-\alpha-\bar{\alpha}}}{(t-s)^{2\bar{\alpha}}}\lesssim_{\hat{\mathcal{M}},T} h(r).
\end{align*}

For  the third term, on the one hand, we
note 
\begin{align*}
\bar{R}^{\mathbb{G}}_{t,s}&=  \mathcal{G}(y_{t})-\mathcal{G}(y_{t-r})- \mathcal{G}(y_{s})+\mathcal{G}(y_{s-r})-D\mathcal{G}(y_{s})\circ \mathcal{G}(y_{s-r})\cdot \delta X_{t,s}+D\mathcal{G}(y_{s-r})\circ \mathcal{G}(y_{s-2r})\cdot \delta X_{t-r,s-r}\\
&= \int^{1}_{0}  [D\mathcal{G}(y_{s}+\tau (y_{t}-y_{s}))-D\mathcal{G}(y_{s})]\textup{d}\tau\circ\mathcal{G}(y_{s-r})\cdot \delta X_{t,s}+   \int^{1}_{0}D\mathcal{G}(y_{s}+\tau(y_{t}-y_{s}))\textup{d}\tau \circ R^{y}_{t,s}\\
&~~~+\int^{1}_{0}  [D\mathcal{G}(y_{s-r}+\tau (y_{t-r}-y_{s-r}))-D\mathcal{G}(y_{s-r})]\textup{d}\tau\circ\mathcal{G}(y_{s-2r})\cdot \delta X_{t-r,s-r}\\
&~~~+   \int^{1}_{0}D\mathcal{G}(y_{s-r}+\tau(y_{t-r}-y_{s-r}))\textup{d}\tau \circ R^{y}_{t-r,s-r}.
\end{align*}
It holds
\begin{align*}
\sup_{\substack{t,s\in[u,v]\\ 0<t-s<r}    }\frac{\|\bar{R}^{\mathbb{G}}_{t,s}\|_{\theta-\sigma_{2}-\alpha-2\bar{\alpha}}}{ (t-s)^{2\bar{\alpha}}    }\lesssim \sup_{\substack{t,s\in[r,T]\\ 0<t-s<r}    }\frac{\|\bar{R}^{\mathbb{G}}_{t,s}\|_{\theta-\sigma_{2}-\alpha-2\bar{\alpha}}}{ (t-s)^{2\alpha}    } (t-s)^{ 2(\alpha-\bar{\alpha})     }
\lesssim_{\hat{\mathcal{M}},T} r^{2(\alpha-\bar{\alpha})}.
\end{align*}
On the other hand, we	rewrite 
\begin{align*}
\bar{R}^{\mathbb{G}}_{t,s}=\sum_{i=1}^{4} \mathbb{K}^{i}_{t,s},
\end{align*}
where
\begin{align*}
\mathbb{K}^{1}_{t,s}&=
\int^{1}_{0}[D\mathcal{G}(y_{t-r}+\tau(y_{t}-y_{t-r}))\circ \mathcal{G}(y_{t-2r})- D\mathcal{G}(y_{s-r})\circ \mathcal{G}(y_{s-2r})]\textup{d}\tau\cdot\delta X_{t,t-r}\\
\mathbb{K}^{2}_{t,s}&=
\int^{1}_{0} [ D\mathcal{G}(y_{s-r}) -D\mathcal{G}(y_{s-r}+\tau(y_{s}-y_{s-r}))]\textup{d}\tau\circ \mathcal{G}(y_{s-2r})\cdot\delta X_{s,s-r}\\
\mathbb{K}^{3}_{t,s}&=  [ D\mathcal{G}(y_{s-r})\circ \mathcal{G}(y_{s-2r})-D\mathcal{G}(y_{s})\circ \mathcal{G}(y_{s-r}) ]\textup{d}\tau\cdot\delta X_{t,s}\\
\mathbb{K}^{4}_{t,s}&=\int^{1}_{0}D\mathcal{G}(y_{t-r}+\tau(y_{t}-y_{t-r}))\textup{d}\tau \circ R^{y}_{t,t-r}-\int^{1}_{0}D\mathcal{G}(y_{s-r}+\tau(y_{s}-y_{s-r}))\textup{d}\tau \circ R^{y}_{s,s-r}.
\end{align*}
For $i=1,\cdots,4$, it holds 
\begin{align*}
\sup_{\substack{t,s\in[u,v]\\ t-s\geq r}    }\frac{\|\mathbb{K}^{i}_{t,s}\|_{\theta-\sigma_{2}-\alpha-2\bar{\alpha}}}{ (t-s)^{2\bar{\alpha}}    }\lesssim_{\hat{\mathcal{M}},T} r^{\alpha-\bar{\alpha}}.
\end{align*}		
Then, we obtain 
\begin{align*}
\sup_{t,l,s\in[u,v]} \frac{\|\delta X_{t,l} \cdot\bar{R}^{\mathbb{G}}_{l,s}\|_{\theta-\sigma_{2}-\alpha-2\bar{\alpha}}}{ (t-l)^{\bar{\alpha}}(l-s)^{2\bar{\alpha}} }\lesssim_{\hat{\mathcal{M}},T}r^{\alpha-\bar{\alpha}}.
\end{align*}

For the fourth term, we write 
\begin{align*}
&~~~~\delta \mathbb{G}^{\prime}_{l,s}:\mathbb{X}_{t,l}+   \delta \bar{\mathbb{G}}^{\prime}_{l,s}:\mathbb{X}(-r)_{t,l}\\
&=[D\mathcal{G}(y_{l})\circ \mathcal{G}(y_{l-r})- D\mathcal{G}(y_{s})\circ \mathcal{G}(y_{s-r})]:\mathbb{X}_{t,l}-[D\mathcal{G}(y_{l-r})\circ \mathcal{G}(y_{l-2r})- D\mathcal{G}(y_{s-r})\circ \mathcal{G}(y_{s-2r})]:\mathbb{X}(-r)_{t,l}.
\end{align*}
Using the similar  arguments of the second and third terms, it holds 
\begin{align*}
\sup_{t,l,s\in[u,v]}\frac{\|\delta \mathbb{G}^{\prime}_{l,s}:\mathbb{X}_{t,l}+   \delta \bar{\mathbb{G}}^{\prime}_{l,s}:\mathbb{X}(-r)_{t,l}  \|_{\theta-\sigma_{2}-\alpha-2\bar{\alpha}   }        } { (t-l)^{2\bar{\alpha}}(l-s)^{\bar{\alpha}}}   \lesssim_{   \hat{\mathcal{M}},T}r^{\alpha-\bar{\alpha}}+ h(r).
\end{align*}
Combining the above estimates, we  obtain there exists a suitable constant $c_{1}$  such that
\begin{align*}
\|I^{1}_{t,s}\|_{\theta-\alpha}\lesssim_{\hat{\mathcal{M}},T}r^{c_{1}}+ h(r).
\end{align*}

Now, it is turn to consider $I^{2}_{t,s}$. 
In fact, we have 
\begin{align*}
\|S_{t,s}\mathbb{G}_{s}\cdot \delta X_{t,s}\|_{\theta-\alpha}\lesssim\|\mathbb{G}_{s}\|_{\theta-2\alpha} \lesssim_{\hat{\mathcal{M}},T} r^{\alpha}.
\end{align*}
And,  using (\ref{eq711}), we derive 
\begin{align*}
&~~~~\|S_{t,s}\big\{\mathbb{G}'_{s}:\mathbb{X}_{t,s}+\bar{\mathbb{G}}'_{s}:\mathbb{X}(-r)_{t,s}\big\}\|_{\theta-\alpha}\\
&\lesssim \|\mathbb{G}'_{s}:\mathbb{X}_{t,s}+\bar{\mathbb{G}}'_{s}:\mathbb{X}(-r)_{t,s}\|_{\theta-\alpha-2\bar{\alpha}}(t-s)^{-2\bar{\alpha}}\\
&\lesssim_{\hat{\mathcal{M}},T} r^{\alpha}+h(r).
\end{align*}

Collecting  the all estimates, the proof is completed.
\end{proof}

\begin{Lemma}\label{lem311}
Let $(y,y') \in\mathcal{D}^{2\alpha}_{\mathbf{X},\theta}([0,T])$    be the mild solution of  $(\ref{eq701})$-$(\ref{eq702})$. 
For $r\leq u<v\leq T $,  there exists some constant $c>0$ such that
\begin{align*}
\rho_{2\bar{\alpha},2\alpha, \theta-\alpha,[u,v]}(\int^{\cdot}_{u} S_{\cdot,l}\mathcal{G}(y_{l})\cdot \textup{d}\bar{\mathbf{X}}_{l},\int^{\cdot}_{u} S_{\cdot,l}\mathcal{G}(y_{l-r})\cdot \textup{d}\bar{\mathbf{X}}_{l})\lesssim_{ \hat{\mathcal{M}},T } r^{c}+h(r).
\end{align*}
\end{Lemma}
\begin{proof}
In Lemma \ref{lem107}, we have proven
\begin{align*}
\sup_{t\in[u,v]}	\|\int^{t}_{u}S_{t,s}[\mathcal{G}(y_{s})-\mathcal{G}(y_{s-r})]\cdot\textup{d}\bar{\mathbf{X}}_{s}\|_{\theta-\alpha}\lesssim_{ \hat{\mathcal{M}},T } r^{c}+h(r).
\end{align*}
Next, let us consider other terms. Recall $  \mathbb{G}_{t}=\mathcal{G}(y_{t})-\mathcal{G}(y_{t-r})$.

Firstly, we can show 
\begin{align*}
\sup_{t\in[u,v]}\|\mathbb{G}_{t}\|_{\theta-2\alpha}
\lesssim_{ \hat{\mathcal{M}},T }r^{\alpha}.
\end{align*}

Secondly, similar to the proof of Lemma \ref{lem107}, it follows
\begin{align*}
\sup_{t,s\in[u,v]}\frac{\|\mathbb{G}_{t}-\mathbb{G}_{s}\|_{\theta-3\alpha}}{(t-s)^{\bar{\alpha}}}\lesssim_{\hat{\mathcal{M}},T} r^{\alpha-\bar{\alpha}}.
\end{align*}

It remains to take into account 
\begin{align*}
\mathcal{R}_{t,s}:=\int^{t}_{u} S_{t,l}\mathbb{G}_{l}\cdot \textup{d}\bar{\mathbf{X}}_{l}-\int^{s}_{u} S_{s,l}\mathbb{G}_{l}\cdot \textup{d}\bar{\mathbf{X}}_{l}-\mathbb{G}_{s}\cdot \delta X_{t,s}.
\end{align*}
By analogous proof in Lemmas \ref{lem006} and \ref{lem107}, we claim that 
for $i=1,2$, there exists a suitable constant $c_{1}$ such that 
\begin{align*}
\sup_{ t,s\in[u,v]} \frac{\|\mathcal{R}_{t,s}\|_{\theta-\alpha-i\alpha}}{(t-s)^{i\bar{\alpha}}} \lesssim_{\hat{\mathcal{M}},T}    r^{c_{1}}+h(r).
\end{align*}		

Consequently, we complete the proof by collecting the above estimates.
\end{proof}

\begin{Theorem}\label{the003}
Let $(y,y') \in\mathcal{D}^{2\alpha}_{\mathbf{X},\theta}([0,T])$  and $(z,z')\in\mathcal{D}^{2\alpha}_{\mathbf{X},\theta}([0,T])$ be the mild solution of  $(\ref{eq701})$-$(\ref{eq702})$ and $(\ref{eq703})$-$(\ref{eq705})$, respectively.  Then, there exists some constant $c>0$ such that
\begin{align}\label{eq901}
\rho_{2\bar{\alpha},2\alpha, \theta-\alpha,[0,T]}(y,z)\lesssim_{\hat{\mathcal{M}}} r^{c}+h(r).
\end{align}
\end{Theorem}
\begin{proof}
In Lemma \ref{lem707}, we state 
\begin{align*}
\rho_{2\bar{\alpha},2\alpha, \theta-\alpha,[0,r]}(y,z)\lesssim_{\hat{\mathcal{M}}} r^{c}.
\end{align*}
Next, let us consider \begin{align*}
\rho_{2\bar{\alpha},2\alpha, \theta-\alpha,[s,t]}(y,z)\lesssim_{\hat{\mathcal{M}}} r^{c},
\end{align*}
where $t,s\in [r,T]$ with $0<t-s\leq 1$.

Note 
\begin{align*}
\rho_{2\bar{\alpha},2\alpha, \theta-\alpha,[s,t]}(y,z)\leq \sum_{i=1}^{4}\rho_{i},
\end{align*}
where 
\begin{align*}
\rho_{1}&=\rho_{2\bar{\alpha},2\alpha, \theta-\alpha,[s,t]}(S_{\cdot,s}y_{s},S_{\cdot,s}z_{s})\\
\rho_{2}&=\rho_{2\bar{\alpha},2\alpha, \theta-\alpha,[s,t]}(\int^{\cdot}_{s}S_{\cdot,u}\mathcal{F}(y_{u})\textup{d}u,
\int^{\cdot}_{s}S_{\cdot,u}\mathcal{F}(z_{u})\textup{d}u)\\
\rho_{3}&=\rho_{2\bar{\alpha},2\alpha, \theta-\alpha,[s,t]}(\int^{\cdot}_{s}S_{\cdot,u}\mathcal{G}(y_{u})\cdot\textup{d}\bar{\mathbf{X}}_{u},
\int^{\cdot}_{s}S_{\cdot,u}\mathcal{G}(z_{u})\cdot\textup{d}\bar{\mathbf{X}}_{u})\\
\rho_{4}&=\rho_{2\bar{\alpha},2\alpha, \theta-\alpha,[s,t]}(\int^{\cdot}_{s}S_{\cdot,u}\mathcal{G}(y_{u-r})\cdot\textup{d}\bar{\mathbf{X}}_{u},
\int^{\cdot}_{s}S_{\cdot,u}\mathcal{G}(y_{u})\cdot\textup{d}\bar{\mathbf{X}}_{u}).
\end{align*}
As for $\rho_{1}$, we have
\begin{align*}
\rho_{1}\lesssim \|y_{s}-z_{s}\|_{\theta-\alpha}.
\end{align*}
Set $\lambda_{\star}:=\min\{1-\sigma_{1},1-2\alpha,\alpha-\bar{\alpha}, \frac{\bar{\alpha}   (\alpha-\sigma)}{\alpha}, 3\bar{\alpha}-2\alpha-\sigma_{2}\}$.
Similar to the proof of Lemma \ref{lem007}, it is not hard to deduce 
\begin{align*}
\rho_{2}\lesssim \rho_{2\bar{\alpha},2\alpha, \theta-\alpha,[s,t]}(y,z)(t-s)^{\lambda_{\star}}.
\end{align*}  
Lemma \ref{lem777} also implies 
\begin{align*}
\rho_{3}\lesssim \rho_{2\bar{\alpha},2\alpha, \theta-\alpha,[s,t]}(y,z)(t-s)^{\lambda_{\star}}+  \|y_{s}-z_{s}\|_{\theta-\alpha}.
\end{align*}  
From Lemma \ref{lem311}, we know 
\begin{align*}
\rho_{4}\lesssim_{\hat{\mathcal{M}},T} r^{c}+h(r).
\end{align*} 
Combining the above estimates, we obtain
\begin{align*}
\rho_{2\bar{\alpha},2\alpha, \theta-\alpha,[s,t]}(y,z)\lesssim _{\hat{\mathcal{M}},T} \rho_{2\bar{\alpha},2\alpha, \theta-\alpha,[s,t]}(y,z)(t-s)^{\lambda_{\star}}+  \|y_{s}-z_{s}\|_{\theta-\alpha}+r^{c}+h(r).
\end{align*}
If we choose  $(t-s)$ such that $   (t-s)^{\lambda_{\star}}<\frac{1}{2}$, then it holds
\begin{align*}
\rho_{2\bar{\alpha},2\alpha, \theta-\alpha,[s,t]}(y,z)\lesssim _{\hat{\mathcal{M}},T}  \|y_{s}-z_{s}\|_{\theta-\alpha}+r^{c}+h(r).
\end{align*}
We emphasize that $t-s$ does not depends on the delay $r$.
Then, for $(s-r)^{\lambda_{\star}}<\frac{1}{2}$, we employ Lemma \ref{lem707} to show 
\begin{align*}
\rho_{2\bar{\alpha},2\alpha, \theta-\alpha,[r,s]}(y,z)&\lesssim _{\hat{\mathcal{M}},T}  \|y_{r}-z_{r}\|_{\theta-\alpha}+r^{c}+h(r)\\
&\lesssim_{\hat{\mathcal{M}},T}  r^{c}+h(r).
\end{align*}
Thus, we obtain 
\begin{align*}
	\|y_{s}-z_{s}\|_{\theta-\alpha}\leq_{\hat{\mathcal{M}},T}  r^{c}+h(r).
\end{align*}
Repeating this process, we can obtain that if   $   (t-s)^{\lambda_{\star}}<\frac{1}{2}$, it holds
\begin{align*}
\rho_{2\bar{\alpha},2\alpha, \theta-\alpha,[s,t]}(y,z)\lesssim _{\hat{\mathcal{M}},T}  r^{c}+h(r).
\end{align*}
As a consequence, collecting the estimate on each interval, we obtain (\ref{eq901}) as desired.
\end{proof}

\section{Example}
In this section, we introduce some equations that fall within the scope of our main results. 

Let $B=\{B_{t}:B_{t}=(B^{1}_{t}, B^{2}_{t}, \cdots, B^{d}_{t}), t\geq 0\}$
is a $d$-dimensional fractional Brownian motion with Hurst parameter $\alpha\in(\frac{1}{3},\frac{1}{2})$. 

We  use Russo-Vallois symmetric approximations (Section 5 in \cite{MR2453555}) to
define the integrals
\begin{align}
&\mathbb{B}^{i,j}_{t,s}=\int^{t}_{s}\textup{d}^{\circ}B^{i}_{u}\int^{u}_{s}\textup{d}^{\circ}B^{j}_{l},~~\textup{for}~i, j\in\{1,2, \cdots, d\}, \label{eq961}\\
&\mathbb{B}(-r)^{i,j}_{t,s}=\int^{t}_{s}\textup{d}^{\circ}B^{i}_{u}\int^{u-r}_{s-r}\textup{d}^{\circ}B^{j}_{l} ,~~\textup{for}~i, j\in\{1,2, \cdots, d\}. \label{eq962}
\end{align}
In Proposition 5.2 in  \cite{MR2453555} , the authors proved that $\bar{\mathbf{B}}=(B, \mathbb{B},\mathbb{B}(-r))$ satisfies Definition \ref{def1}.

Let $W=\{W_{t}:W_{t}=(W^{1}_{t}, W^{2}_{t}, \cdots, W^{d}_{t}), t\geq 0\}$
is a $d$-dimensional Brownian motion.

We use Stratonovich integral to define 
\begin{align}
&\mathbb{W}^{i,j}_{t,s}=\int^{t}_{s}\textup{d}^{\circ}W^{i}_{u}\int^{u}_{s}\textup{d}^{\circ}W^{j}_{l}-\frac{1}{2}(t-s)\delta_{i,j},~~\textup{for}~i, j\in\{1,2, \cdots, d\}, \label{eq963}\\
&\mathbb{W}(-r)^{i,j}_{t,s}=\int^{t}_{s}\textup{d}^{\circ}W^{i}_{u}\int^{u-r}_{s-r}\textup{d}^{\circ}W^{j}_{l} ,~~\textup{for}~i, j\in\{1,2, \cdots, d\},\label{eq964}
\end{align}
$\delta_{i,j}$ represents Kronecker function.

It is easy to check $\bar{\mathbf{W}}=(W, \mathbb{W},\mathbb{W}(-r))$ also satisfies Definition \ref{def1} in this case. Moreover, we state
that for  $0<\tilde{\alpha}<\alpha$,  it holds
\begin{align}\label{eq991}
\lim_{r\rightarrow 0}\|\mathbb{W}-\mathbb{W}(-r)\|_{2\tilde{\alpha},\mathbb{R}^{d}\times \mathbb{R}^{d}}=0.
\end{align} 
In fact, the proof of (\ref{eq991}) is the slight modification of that of Lemma 7.2 in \cite{MR4117091}.

Set $ p>1, r>0, \sigma\in (\frac{\alpha}{2}, \alpha), \tilde{\alpha}\in (\sigma, \alpha)$. Let $\mathbb{T}^{d}$  be the $d$-dimensional torus. For $\theta\in \mathbb{R}$, introduce $\mathcal{B}_{\theta} = H^{2\theta,p}(\mathbb{T}^{d})$, where $H^{2\theta,p}(\mathbb{T}^{d})$ are Bessel potential spaces.

{\bf Existence of the global solution}

Consider the following  delay rough partial differential equation
\begin{align}
\textup{d}y&=[\Delta y_{t}+F(y_{t},y_{t-r})]    	\textup{d}t+  [(-\Delta)^{\sigma} y_{t}+   (-\Delta)^{\sigma} y_{t-r}]\cdot\textup{d}\bar{\mathbf{B}}_{t},\label{eq971}\\
y_{t}&=\phi_{t}, ~~t\in[-r,0], \label{eq972}
\end{align} 
where $\bar{\mathbf{B}}_{t}$ is defined by either (\ref{eq961})-(\ref{eq962}) or  (\ref{eq963})-(\ref{eq964}), $F$ is a  nonlinear operator satisfying $\mathbf{H}_{3}$, $(\phi,\phi')$ is a controlled path based on $\mathbf{B}$ belonging to $\mathcal{D}^{2\tilde{\alpha}}_{\mathbf{B},\theta}([-r,0])\cap\mathcal{D}^{2\alpha}_{\mathbf{B},\theta}([-r,0])$. We note $G(y_{t},y_{t-r}):= (-\Delta)^{\sigma} y_{t}+   (-\Delta)^{\sigma} y_{t-r}$ is an operator from $B^{2}_{\hat{\theta}}$ to $B_{\hat{\theta}-\sigma}$ for every $\hat{\theta}\in \mathbb{R}$.
Therefore, applying Theorem\ref{Cor002}, we know that there exists a global solution $(y,y')\in \mathcal{D}^{2\tilde{\alpha}}_{\mathbf{B},\theta}([0,T])\cap\mathcal{D}^{2\alpha}_{\mathbf{B},\theta}([0,T]) $  for (\ref{eq971})-(\ref{eq972}).

{\bf Convergence with respect to the delay}

Consider the following  two  rough partial differential equations
\begin{align}
\textup{d}y&=[\Delta y_{t}+\mathcal{F}(y_{t})]    	\textup{d}t+  (-\Delta)^{\sigma} y_{t-r} \cdot\textup{d}\bar{\mathbf{W}}_{t},\label{eq973}\\
y_{t}&=\psi_{t}, ~~t\in[-r,0], \label{eq974}
\end{align} 
and 
\begin{align}
\textup{d}z&=[\Delta z_{t}+\mathcal{F}(z_{t})] \textup{d}t+  (-\Delta)^{\sigma} z_{t}\cdot\textup{d}\mathbf{W}_{t},\label{eq975}\\
z_{0}&=\psi_{0}, \label{eq976}
\end{align} 
where $\bar{\mathbf{W}}_{t}$ and $\mathbf{W}_{t}$ can be  defined by  (\ref{eq963})-(\ref{eq964}), $\mathcal{F}$ is a  nonlinear operator satisfying $\hat{\mathbf{H}}_{3}$, $(\psi,\psi')$ is a controlled path based on $\mathbf{W}$ belonging to $\mathcal{D}^{2\tilde{\alpha}}_{\mathbf{B},\theta}([-r_{0},0])\cap\mathcal{D}^{2\alpha}_{\mathbf{B},\theta}([-r_{0},0])$ with $r_{0}\geq r$. From Theorem \ref{Cor002}, we can obtain the existence of the global solutions of (\ref{eq973})-(\ref{eq974}) and  (\ref{eq975})-(\ref{eq976}). Taking  $\bar{\alpha}\in (\frac{4\alpha}{5},\alpha )$  satisfying $ \sigma_{2}+2\bar{\alpha}-2\alpha\geq 0$ and  $3\bar{\alpha}-2\alpha-\sigma_{2}\geq 0$, it follows from  Theorem \ref{the003} that
\begin{align*}
\lim_{r\rightarrow 0}\rho_{2\bar{\alpha},2\alpha, \theta-\alpha,[0,T]}(y,z)=0.
\end{align*} 
\section{Appendix}
\subsection{A.1}\label{A.1}
	The followings are devoted to the definition of rough path \cite{MR3957994,MR4040992}. 

\begin{Definition}\label{A.11}
	[$\alpha$-Hölder path] Let $X:\mathbb{R}\rightarrow \mathbb{R}^{d}$ be a locally $\alpha$-H\"older path. For a compact interval $I \subset \mathbb{R}$, $\alpha\in (\frac{1}{3},\frac{1}{2}]$, we call a pair  $\mathbf{X}=\left(X, \mathbb{X}\right)$  $\alpha$-Hölder rough path  if $X\in \mathcal{C}^{\alpha}_{1}(I,\mathbb{R}^d)$ and $\mathbb{X}\in \mathcal{C}^{2 \alpha}_{2}\left(I,\mathbb{R}^d \times \mathbb{R}^d\right)$, and the following  Chen's relation holds:
	\begin{align*}
		\mathbb{X}^{i,j}_{t,s}-\mathbb{X}^{i,j}_{u,s}-\mathbb{X}^{i,j}_{t,u}=\delta X^{i}_{u,s}\otimes \delta X^{j}_{t,u},
	\end{align*}
	for  $t,u,s \in \Delta_{3,I}$, and $1\leq i, j\leq d$.
\end{Definition}

For $\alpha\in (\frac{1}{3},\frac{1}{2}]$, denote  by $\mathcal{C}^{\alpha}(I,\mathbb{R}^d)$ the space of $\alpha$-H\"older rough paths  $\mathbf{X}=(X, \mathbb{X})$ with respect to the interval $I$.
For two $\alpha$-Hölder rough paths $\mathbf{X}, \mathbf{Y}\in \mathcal{C}^{\alpha}(I, \mathbb{R}^d)$,  the metric of $\mathcal{C}^{\alpha}(I, \mathbb{R}^d)$ is defined by 
\begin{align*}
	\rho_{\alpha,I}(\mathbf{X},\mathbf{Y})=\|\delta(X-Y)\|_{\alpha,\mathbb{R}^{d}}+\|\mathbb{X}-\mathbb{Y}\|_{2\alpha,\mathbb{R}^{d}\times \mathbb{R}^{d}}.
\end{align*}
And, set $\rho_{\alpha,I}(\mathbf{X})=	\rho_{\alpha,I}(\mathbf{X},0)$.

\begin{Definition}[Controlled  path] \label{def001}
	Let  $\mathbf{X}=(X, \mathbb{X})\in\mathcal{C}^{\alpha}(I, \mathbb{R}^d)$ with $I$. We call a pair $(y,y')$  a controlled  path based on $\mathbf{X}$ if $y\in\mathcal{C}_{1}(I,\mathcal{B}_{\theta})$ and $y'\in (\mathcal{C}^{0,\alpha}_{\theta-\alpha}(I))^{d}$  for $\theta\in\mathbb{R}$ such that 
	\begin{align*}
		R^{y}_{t,s}:=\delta y_{t,s}-y'_{s}\cdot \delta X_{t,s}=\delta y_{t,s}-\sum^{d}_{i=1}y'^{,i}_{s}\delta X^{i}_{t,s},~\text{for}~t,s \in \Delta_{2,I},
	\end{align*}
	belongs to $\mathcal{C}_{\theta}^{\alpha,2\alpha}(I)$.
\end{Definition}

We denote by $\mathcal{D}^{2\alpha}_{\mathbf{X},\theta} (I)$ the space of controlled rough paths based on $\mathbf{X}$ with $I$, and endow the norm:
\begin{align*}
	\|y,y'\|_{\mathbf{X},2\alpha,\theta}:=\|y\|_{\infty,\theta}+\|y'\|_{\infty,\theta-\alpha}+\|y'\|_{\alpha,\theta-2\alpha}+\|R^{y}\|_{\alpha,\theta-\alpha}+\|R^{y}\|_{2\alpha,\theta-2\alpha},
\end{align*}
where we use $\|y'\|_{\infty,\theta-\alpha}:=\sum^{d}_{i=1} \|y'^{,i}\|_{\infty,\theta-\alpha} $, and  $\|y'\|_{\alpha,\theta-2\alpha}:=\sum^{d}_{i=1} \|y'^{,i}\|_{\alpha,\theta-2\alpha}$ for simplicity.

For $y$ in Definition \ref{def001}, it holds:
\begin{align}
	\|\delta y\|_{\alpha,\theta-\alpha}&\leq \|y'\|_{\infty,\theta-\alpha}\|\delta X\|_{\alpha, \mathbb{R}^{d}}+  	\|\delta R^{y}\|_{\alpha,\theta-\alpha}\notag\\
	&\leq (1+\rho_{\alpha, I}(\mathbf{X})) \|y,y'\|_{\mathbf{X},2\alpha,\theta}.\label{eq100}
\end{align}

	For two rough path $(y,y')$ and $(z,z')$, we introduce the distance:
\begin{align*}
	\rho_{2\tilde{\alpha},2\alpha, \theta}(y,z):&= \|y-z\|_{\infty,\theta}+\|y'-z'\|_{\infty,\theta-\alpha}+\|\delta(y'-z')\|_{\tilde{\alpha},\theta-2\alpha}+\|R^{y}-R^{z}\|_{\tilde{\alpha},\theta-\alpha}\\
	&~~~+ \|R^{y}-R^{z}\|_{2\tilde{\alpha},\theta-2\alpha},
\end{align*}
where $\tilde{\alpha},\alpha,\theta \in\mathbb{R}$.

\subsection{A.2}\label{A.2}

\begin{Proof}  Lemma \ref{lem001}.
	
	Due to $G\in \mathcal{C}^{2}_{\theta-2\alpha,-\sigma}(\mathcal{B}^{2},\mathcal{B})$, it is easy to see that 
	\begin{align*}
		\|m\|_{\infty,\theta-\sigma}&\lesssim_{G} (\|y\|_{\infty,\theta}+\|z\|_{\infty,\theta}),\\
		\|m'\|_{\infty,\theta-\sigma-\alpha}&\lesssim_{G} \|y'\|_{\infty,\theta-\alpha}, \\
		\|\bar{m}'\|_{\infty,\theta-\sigma-\alpha}&\lesssim_{G} \|z'\|_{\infty,\theta-\alpha}, 
	\end{align*}
	which means
	$(m, m',\bar{m}')\in  \mathcal{C}_{1}(I_{1},\mathcal{B}_{\theta-\sigma})\times (\mathcal{C}_{1}(I_{1},\mathcal{B}_{\theta-\sigma-\alpha}))^{d}\times(\mathcal{C}_{1}(I_{1},\mathcal{B}_{\theta-\sigma-\alpha}))^{d}$.

	We proceed to prove that $m',\bar{m}'$ still belong to $(\mathcal{C}_{1}^{\alpha}(I,\mathcal{B}_{\theta-\sigma-2\alpha}))^{d}.$ For $t,s\in [0,T]$, notice that 
	\begin{align*}
		\delta m'_{t,s}&= D_{x_{1}}G(y_{t},z_{t-r})\circ y'_{t}-D_{x_{1}}G(y_{s},z_{s-r})\circ y'_{s}\\
		&=   (D_{x_{1}}G(y_{t},z_{t-r})-D_{x_{1}}G(y_{s},z_{s-r})) \circ y'_{t}+     D_{x_{1}}G(y_{s},z_{s-r}) \circ \delta y'_{t,s}.
	\end{align*}
	Taylor formula yields that
	\begin{align*}
		\|\delta m'\|_{\alpha, \theta-\sigma-2\alpha}&\lesssim_{G}(\|\delta y\|_{\alpha, \theta-2\alpha}+ \|\delta z\|_{\alpha, \theta-2\alpha})\|y'\|_{\infty,\theta-\alpha}+\|\delta y'\|_{\alpha,\theta-2\alpha}\\
		&\lesssim_{G}[(1+\rho_{\alpha, I_{1}}(\mathbf{X}))\|y,y'\|_{\mathbf{X},2\alpha,\theta}+
		(1+\rho_{\alpha, I_{2}}(\mathbf{X}))\|z,z'\|_{\mathbf{X},2\alpha,\theta}]\times (1+ \|y,y'\|_{\mathbf{X},2\alpha,\theta} ),
	\end{align*}
	where the first inequality holds since $DG$ and $D^{2} G$ are bounded operators, and the second one follows from (\ref{eq100}). By similar arguments, we can prove  $\bar{m}'\in (\mathcal{C}_{1}^{\alpha}(I_{1},\mathcal{B}_{\theta-\sigma-2\alpha}))^{d}$. 
	
	The next task is to show 
	\begin{align*}
		\bar{R}^{m}_{t,s}:&=G(y_{t},z_{t-r})-G(y_{s},z_{s-r})-DG_{x_{1}}(y_{s},z_{s-r})\circ y'_{s}\cdot\delta X_{t,s}-D_{x_{2}}G(y_{s},z_{s-r})\circ z'_{s-r}\cdot\delta X_{t-r,s-r}
	\end{align*}
	belongs to $\mathcal{C}_{\theta-\sigma}^{\alpha,2\alpha}(I_{1})$.
	
	Notice that
	\begin{align*}
		\bar{R}^{m}_{t,s}&=  G(y_{t},z_{t-r})-G(y_{s},z_{s-r})-D_{x_{1}}G(y_{s},z_{s-r})\circ \delta y_{t,s}+D_{x_{1}}G(y_{s},z_{s-r})\circ R^{y}_{t,s}\\
		&~~~-D_{x_{2}}G(y_{s},z_{s-r})\circ \delta z_{t-r,s-r}+D_{x_{2}}G(y_{s},z_{s-r})\circ R^{z}_{t-r,s-r}\\
		&=\sum_{j=1}^{4} K^{j}_{t,s},
	\end{align*}
	where
	\begin{align*}
		K^{1,m}_{t,s}&= [\int^{1}_{0}\int^{1}_{0} D^{2}_{x^{2}_{1}} G(y_{s}+\tau \tilde{\tau}\delta y_{t,s}, z_{s-r}+\tau \tilde{\tau}\delta z_{t-r,s-r})\tau\textup{d}\tilde{\tau}\textup{d}\tau]\circ (\delta y_{t,s}\otimes \delta y_{t,s}),\\
		K^{2,m}_{t,s}&= 2[\int^{1}_{0}\int^{1}_{0} D^{2}_{x_{1}x_{2}} G(y_{s}+\tau \tilde{\tau}\delta y_{t,s}, z_{s-r}+\tau \tilde{\tau}\delta z_{t-r,s-r})\tau\textup{d}\tilde{\tau}\textup{d}\tau]\circ (\delta z_{t-r,s-r}\otimes \delta y_{t,s}),\\
		K^{3,m}_{t,s}&= [\int^{1}_{0}\int^{1}_{0} D^{2}_{x^{2}_{2}} G(y_{s}+\tau \tilde{\tau}\delta y_{t,s}, z_{s-r}+\tau \tilde{\tau}\delta z_{t-r,s-r})\tau\textup{d}\tilde{\tau}\textup{d}\tau]\circ (\delta z_{t-r,s-r}\otimes \delta z_{t-r,s-r}),\\
		K^{4,m}_{t,s}&=D_{x_{1}}G(y_{s},z_{s-r})\circ R^{y}_{t,s}+D_{x_{2}}G(y_{s},z_{s-r})\circ R^{z}_{t-r,s-r}.
	\end{align*}
	
	Then, for $i=1,2$, we derive 
	\begin{align*}
		\|K^{1,m}\|_{i\alpha, \theta-\sigma-i\alpha}&\lesssim_{G}\|\delta y\|^{2}_{\alpha, \theta-\alpha},\\
		\|K^{2,m}\|_{i\alpha, \theta-\sigma-i\alpha}&\lesssim_{G}\|\delta y\|_{\alpha, \theta-\alpha}\|\delta z\|_{\alpha, \theta-\alpha},\\
		\|K^{3,m}\|_{i\alpha, \theta-\sigma-i\alpha}&\lesssim_{G}\|\delta z\|^{2}_{\alpha, \theta-\alpha},\\
		\|K^{4,m}\|_{i\alpha, \theta-\sigma-i\alpha}&\lesssim_{G}\|R^{y}\|_{i\alpha, \theta-i\alpha}+\|R^{z}\|_{i\alpha, \theta-i\alpha}.
	\end{align*}	
	Thus, for $i=1,2$, we have obtain
	\begin{align*}
		\|\bar{R}^{m}\|_{i\alpha, \theta-\sigma-i\alpha}&\lesssim_{G}  (1+\rho_{\alpha, I_{1}}(\mathbf{X})   +  \rho_{\alpha, I_{2}}(\mathbf{X}) )^{2}(1+\|y,y'\|_{\mathbf{X},2\alpha,\theta}+
		\|z,z'\|_{\mathbf{X},2\alpha,\theta})^{2}.
	\end{align*}
	Combining the above estimates,  we obtain $(m, m',\bar{m}')\in \mathcal{D}^{2\alpha}_{\bar{\mathbf{X}},\theta-\sigma}(I_{1})$ and (\ref{eq101}). 
\end{Proof}

	\begin{Proof}  Corollary \ref{cor001}.
		
	Here, we just need to estimate those terms which cause the quadratic terms of $\|y,y'\|_{\mathbf{X},2\alpha,\theta}$. 
	
	Note 
	\begin{align*}
		\|\delta m'_{t,s}\|_{\alpha,\theta-\sigma-2\alpha}&=  \|D_{x_{1}}G(y_{t},z_{t-r})\circ y'_{t}-D_{x_{1}}G(y_{s},z_{s-r})\circ y'_{s}\| _{\alpha,\theta-\sigma-2\alpha}\\
		&\lesssim_{G} \|\delta y_{t,s}\|_{\alpha,\theta-\alpha}+\| \delta z_{t-r,s-r}\|_{\alpha,\theta-\alpha}.
	\end{align*}
	
	Rewrite 
	\begin{align*}
		\bar{R}^{m}_{t,s}=\sum^{4}_{j=1} \tilde{K}^{j,m}_{t,s},
	\end{align*}
	where  
	\begin{align*}
		\tilde{K}^{1,m}_{t,s}&=\int^{1}_{0} [D_{x_{1}} G( y_{s}+\tau(\delta y_{t,s}), z_{s-r}+\tau(\delta z_{t-r,s-r}) )-D_{x_{1}} G( y_{s}, z_{s-r} )]\textup{d}\tau\circ G(y_{s},z_{s-r})\cdot \delta X_{t,s},\\
		\tilde{K}^{2,m}_{t,s}&=\int^{1}_{0} D_{x_{1}} G( y_{s}, z_{s-r} )\textup{d}\tau\circ R^{y}_{t,s},\\
		\tilde{K}^{3,m}_{t,s}&=\int^{1}_{0} [D_{x_{2}} G( y_{s}+\tau(\delta y_{t,s}), z_{s-r}+\tau(\delta z_{t-r,s-r}) )-D_{x_{2}} G( y_{s}, z_{s-r} )]\textup{d}\tau\circ z'_{s-r}\cdot \delta X_{t-r,s-r},\\
		\tilde{K}^{4,m}_{t,s}&=\int^{1}_{0} D_{x_{2}} G( y_{s}, z_{s-r} )\textup{d}\tau\circ R^{z}_{t,s}.
	\end{align*}
	Then, we obtain 
	\begin{align*}
		\|\tilde{K}^{1,m}\|_{\alpha,\theta-\sigma-\alpha}\lesssim_{G} \rho_{\alpha, I_{1}}(\mathbf{X})(\|y\|_{\infty,\theta}+\|z\|_{\infty, \theta}),
	\end{align*}
	and 
	\begin{align*}
		\|\tilde{K}^{1,m}\|_{2\alpha,\theta-\sigma-2\alpha}\lesssim_{G} \rho_{\alpha, I_{1}}(\mathbf{X})(\|\delta y\|_{\alpha,\theta-\alpha}+\|\delta z\|_{\alpha,\theta-\alpha}).
	\end{align*}
	By similar arguments in Lemma \ref{lem001}, we can prove other terms do not cause the quadratic terms of $\|y,y'\|_{\mathbf{X},2\alpha,\theta}$ . 
\end{Proof}

	\begin{Proof} Lemma \ref{lem005}.
		
	Firstly, by Taylor formula, it is easy to obtain
	\begin{align*}
		\|m-l\|_{\infty,\theta-\sigma}&\lesssim_{G} \|y-u\|_{\infty,\theta}+\|z-v\|_{\infty,\theta},\\
		\|m'-l'\|_{\infty,\theta-\sigma-\alpha}&\lesssim_{G} (\|y-u\|_{\infty,\theta}+\|z-v\|_{\infty,\theta})\|y'\|_{\infty,\theta-\alpha}+ \|y'-u'\|_{\infty,\theta-\alpha},\\
		\|\bar{m}'-\bar{l}'\|_{\infty,\theta-\sigma-\alpha}&\lesssim_{G} (\|y-u\|_{\infty,\theta}+\|z-v\|_{\infty,\theta})\|z'\|_{\infty,\theta-\alpha}+ \|z'-v'\|_{\infty,\theta-\alpha}.
	\end{align*}
	
	Next, let us estimate $\|\delta(m'-l')\|_{\alpha, \theta-\sigma-2\alpha}$. In fact, we have
	\begin{align*}
		&~~~\delta(m'-l')_{t,s}\\
		&=D_{x_{1}}G(y_{t},z_{t-r})\circ y'_{t}-D_{x_{1}}G(y_{s},z_{s-r})\circ y'_{s}-D_{x_{1}}G(u_{t},v_{t-r})\circ u'_{t}+D_{x_{1}}G(u_{s},v_{s-r})\circ u'_{s}\\
		&=\int^{1}_{0}[D^{2}_{x^{2}_{1}}G(y_{s}+\tau \delta y_{t,s}, z_{s-r}+\tau \delta z_{t-r,s-r})
		-D^{2}_{x^{2}_{1}}G(u_{s}+\tau \delta u_{t,s}, v_{s-r}+\tau \delta v_{t-r,s-r})]\textup{d}\tau\circ(\delta y_{t,s}\otimes y'_{t})\\
		&~~~+\int^{1}_{0}[D^{2}_{x_{1}x_{2}}G(y_{s}+\tau \delta y_{t,s}, z_{s-r}+\tau \delta z_{t-r,s-r})\\
		&~~~~~~~~~~~-D^{2}_{x_{1}x_{2}}G(u_{s}+\tau \delta u_{t,s}, v_{s-r}+\tau \delta v_{t-r,s-r})]\textup{d}\tau\circ(\delta z_{t-r,s-r}\otimes y'_{t})\\
		&~~~+\int^{1}_{0} D^{2}_{x^{2}_{1}}G(u_{s}+\tau \delta u_{t,s}, v_{s-r}+\tau \delta v_{t-r,s-r})\textup{d}\tau\circ[( \delta y_{t,s}-\delta u_{t,s})\otimes y'_{t}]\\
		&~~~+\int^{1}_{0} D^{2}_{x_{1}x_{2}}G(u_{s}+\tau \delta u_{t,s}, v_{s-r}+\tau \delta v_{t-r,s-r})\textup{d}\tau\circ[( \delta z_{t,s}-\delta v_{t-r,s-r})\otimes y'_{t}]\\
		&~~~+[   D_{x_{1}}G(u_{t},v_{t-r})-D_{x_{1}}G(u_{s},v_{s-r})]\circ(y'_{t}-u'_{t})\\
		&~~~+[ D_{x_{1}}G(y_{s},z_{s-r})-D_{x_{1}}G(u_{s},v_{s-r})]\circ \delta y'_{t,s}
		+D_{x_{1}}G(u_{s},v_{s-r})\circ \delta (y'-u')_{t,s}.
	\end{align*}
	After calculations, it follows
	\begin{align*}
		&~~~\|\delta(m'-l')\|_{\alpha, \theta-\sigma-2\alpha}\\	&\lesssim_{G} (1+\rho_{\alpha, I_{1}}(\mathbf{X})+\rho_{\alpha, I_{2}}(\mathbf{X}) ) (\|y-u,y'-u'\|_{\mathbf{X},2\alpha,\theta}+\|z-v,z'-v'\|_{\mathbf{X},2\alpha,\theta}    )\\
		&~~~\times(1+\|y,y'\|_{\mathbf{X},2\alpha,\theta}+
		\|z,z'\|_{\mathbf{X},2\alpha,\theta}+\|u,u'\|_{\mathbf{X},2\alpha,\theta}+\|v,v'\|_{\mathbf{X},2\alpha,\theta})^{2}.
	\end{align*}
	Analogously, we can also estimate $\|\delta(\bar{m}'-\bar{l}')\|_{\alpha, \theta-\sigma-2\alpha}$.
	
	Finally, consider 
	\begin{align*}
		{\bar{R}}^{m-l}_{t,s}:=(m-l)_{t,s}+(m'-l')_{s}\cdot\delta X_{t,s}+(\bar{m}'-\bar{l}')_{s}\cdot\delta X_{t-r,s-r}.
	\end{align*}
	Similar to the decomposition of $ {\bar{R}}^{m}$, we also have
	\begin{align*}
		\|{\bar{R}}^{m-l}\|_{t,s}=\sum_{j=1}^{4}(K^{j,m}-K^{j,l})_{t,s}. 
	\end{align*}	
	By the further  arguments, for $i=1,2$, we obtain 
	\begin{align*}
		&~~~\|{\bar{R}}^{m-l}\|_{i\alpha,\theta-\sigma-i \alpha}\\
		&\lesssim_{G}(1+\rho_{\alpha, I_{1}}(\mathbf{X})+\rho_{\alpha, I_{2}}(\mathbf{X}))^{2} (\|y-u,y'-u'\|_{\mathbf{X},2\alpha,\theta}+\|z-v,z'-v'\|_{\mathbf{X},2\alpha,\theta}    )\\
		&~~~\times(1+\|y,y'\|_{\mathbf{X},2\alpha,\theta}+
		\|z,z'\|_{\mathbf{X},2\alpha,\theta}+\|u,u'\|_{\mathbf{X},2\alpha,\theta}+\|v,v'\|_{\mathbf{X},2\alpha,\theta})^{2}. 
	\end{align*}	
	Combining the above estimates, we finish the proof.
\end{Proof}
\begin{Proof} Lemma \ref{lem003}.
	
	For $t,s\in [0,T]$, set $f_{t,s}=y_{s}\cdot\delta X_{t,s}+y^{\prime}_{s}:\mathbb{X}_{t,s}+   \bar{y}^{\prime}_{s}:\mathbb{X}(-r)_{t,s}$. If we can prove $f\in \mathcal{J}^{\alpha}_{\theta}(I)$, then the proof is completed by Lemma \ref{lem002}.  Let us achieve it in the followings. 
	
	It is easy to see that 
	\begin{align*}
		\|y\cdot \delta X\|_{\alpha,\theta}+\|y^{\prime}:\mathbb{X}\|_{2\alpha,\theta-\alpha  }+   \|\bar{y}^{\prime}:\mathbb{X}(-r)\|_{2\alpha,\theta-\alpha  } \lesssim_{G} \rho_{\alpha,I}(\bar{\mathbf{X}}) 	\|y,y',\bar{y}'\|_{\bar{\mathbf{X}},2\alpha,\theta}.
	\end{align*}
	which means that $y\cdot \delta X\in \mathcal{C}_{2}^{\alpha}(I,\mathcal{B}_{\theta}) $ and $  y^{\prime}:\mathbb{X}+ \bar{y}^{\prime}:\mathbb{X}(-r)\in \mathcal{C}^{2\alpha}_{2}(I,\mathcal{B}_{\theta-\alpha})$.    
	
	From chen's relations, we  note
	\begin{align*}
		\delta f_{t,u,s}= -\delta X_{t,u}\cdot\bar{R}^{y}_{u,s}-\delta y'_{u,s}:\mathbb{X}_{t,u}- \delta \bar{y}'_{u,s}:\mathbb{X}(-r)_{t,u}.
	\end{align*}
	Furthermore, we have 
	\begin{align*}
		\frac{\|\delta X \cdot\bar{R}^{y}\|_{\theta-2\alpha}}{(t-u)^{\alpha}(u-s)^{2\alpha}}\leq \|\bar{R}^{y}\|_{2\alpha,\theta-2\alpha}\|\delta X\|_{\alpha,\mathbb{R}^{d}} ,
	\end{align*}
	and 
	\begin{align*}
		&~~~\frac{\|\delta y'_{u,s}:\mathbb{X}_{t,u}+\delta \bar{y}'_{u,s}:\mathbb{X}(-r)_{t,u}\|_{\theta-2\alpha}}{(t-u)^{2\alpha}(u-s)^{\alpha}}\\
		&\leq \|\delta y'\|_{\alpha,\theta-2\alpha}\|\mathbb{X}\|_{2\alpha,\mathbb{R}^{d}}+\|\delta \bar{y}'\|_{\alpha,\theta-2\alpha}\|\mathbb{X}(-r)\|_{2\alpha,\mathbb{R}^{d}}.
	\end{align*}
	It  follows that $\delta X\cdot \bar{R}^{y}\in \mathcal{C}_{3}^{2\alpha,\alpha}(\mathcal{B}_{\theta-2\alpha}) $ and $\delta y':\mathbb{X}+\delta \bar{y}':\mathbb{X}(-r)\in \mathcal{C}_{3}^{\alpha,2\alpha}(\mathcal{B}_{\theta-2\alpha}).$  
	
	Therefore, we have shown $f\in \mathcal{J}^{\alpha}_{\theta}(I)$.	
\end{Proof}
\begin{Proof} Lemma \ref{lem006}.
	
	Since $(i)$ is the immediate result of $(ii)$, we provide the proof of $(ii)$  as follows. 
	
	Firstly, we show $\|\zeta\|_{\infty,\theta+\sigma}<\infty$.
	
	Set 
	\begin{align*}
		\eta_{t,s}=\int^{t}_{s}S_{t,u} y_{u}\cdot\textup{d}\bar{\mathbf{X}}_{u}- S_{t,s}[y_{s}\cdot\delta X_{t,s}+y^{\prime}_{0}:\mathbb{X}_{t,s}+   \bar{y}^{\prime}_{s}:\mathbb{X}(-r)_{t,s}].
	\end{align*}
	Rewrite
	\begin{align*}
		\zeta_{t}=  \eta_{t,0} + S_{t,0}[y_{0}\cdot\delta X_{t,0}+y^{\prime}_{0}:\mathbb{X}_{t,0}+   \bar{y}^{\prime}_{0}:\mathbb{X}(-r)_{t,0}].
	\end{align*}
	Applying Lemma \ref{lem003} to $\eta_{t,0}$ with $\beta=\sigma+2\tilde{\alpha}$, we obtain 
	\begin{align*}
		\|\eta_{t,0}\|_{\theta+\sigma}&\lesssim \rho_{\tilde{\alpha}, I}(\bar{\mathbf{X}})  \|y,y',\bar{y}'\|_{\bar{\mathbf{X}},2\tilde{\alpha},\theta}t^{\tilde{\alpha}-\sigma}\\
		&\lesssim \rho_{\alpha, I}(\bar{\mathbf{X}})  \|y,y',\bar{y}'\|_{\bar{\mathbf{X}},2\tilde{\alpha},\theta}t^{\tilde{\alpha}-\sigma}.
	\end{align*}
	By condition (\ref{eq050}), we can  derive 
	\begin{align*}
		&~~~~\|S_{t,0}[y_{0}\cdot\delta X_{t,0}+y^{\prime}_{0}:\mathbb{X}_{t,0}+   \bar{y}^{\prime}_{0}:\mathbb{X}(-r)_{t,0}]\|_{\theta+\sigma}\\
		&\lesssim t^{-\sigma }\|y_{0}\cdot\delta X_{t,0}\|_{\theta}+t^{-\tilde{\alpha}-\sigma }\|y^{\prime}_{0}:\mathbb{X}_{t,0}\|_{\theta-\tilde{\alpha}}+     t^{-\tilde{\alpha}-\sigma }\|\bar{y}^{\prime}_{0}:\mathbb{X}(-r)_{t,0}\|_{\theta-\tilde{\alpha}}\\
		&\lesssim \rho_{\alpha, I}(\bar{\mathbf{X}})(\|y_0\|_{\theta} + \|y'_0\|_{\theta-\tilde{\alpha}}   +\|\bar{y}'_{0}\|_{\theta-\tilde{\alpha}})t^{\tilde{\alpha}-\sigma}.
	\end{align*}  
	Thus, we know 
	\begin{align}\label{eq105}
		\|\zeta\|_{\infty,\theta+\sigma}\lesssim  \rho_{\alpha, I}(\bar{\mathbf{X}})  \|y,y',\bar{y}'\|_{\bar{\mathbf{X}},2\tilde{\alpha},\theta}T^{\tilde{\alpha}-\sigma}.
	\end{align}
	
	Secondly, let us prove $\|\zeta'\|_{\infty,\theta-\tilde{\alpha}+\sigma}<  \infty$ and $  \|\delta\zeta'\|_{\alpha,\theta-2\tilde{\alpha}+\sigma}<\infty$.

	By (\ref{eq053}), we can derive 
	\begin{align*}
		\|\zeta'\|_{\infty,\theta-\tilde{\alpha}+\sigma}= \|y\|_{\infty,\theta-\tilde{\alpha}+\sigma}&\lesssim 
		\sup_{t\in[0,T]}\frac{\|y_{t}-y_{0}\|_{\theta-\tilde{\alpha}+\sigma}}{t^{\tilde{\alpha}-\sigma}} t^{\tilde{\alpha}-\sigma}+ \|y_{0}\|_{\theta-\tilde{\alpha}+\sigma}\\
		&\lesssim \|y\|_{\infty, \theta}^{\frac{\sigma} {\tilde{\alpha}}}\|\delta y\|_{\tilde{\alpha}-\sigma, \theta-\tilde{\alpha}}^{\frac{\tilde{\alpha}-\sigma} {\tilde{\alpha}}} T^{\tilde{\alpha}-\sigma}+ \|y_{0}\|_{\theta},
	\end{align*}
	which implies 
	\begin{align*}
		\|y\|_{\infty, \theta-\tilde{\alpha}+\sigma}\lesssim \|y,y',\bar{y}'\|_{\bar{\mathbf{X}},2\tilde{\alpha},\theta}T^{\tilde{\alpha}-\sigma}+ \|y_{0}\|_{\theta}.
	\end{align*}
	
	Using (\ref{eq053}) and (\ref{eq103}), we obtain  
	\begin{align}
		&~~~~\|\delta\zeta'\|_{\tilde{\alpha},\theta-2\tilde{\alpha}+\sigma}\notag\\
		&\lesssim (\|y'_{t}\|_{\infty, \theta-2\tilde{\alpha}+\sigma}+\|\bar{y}'_{t}\|_{\infty, \theta-2\tilde{\alpha}+\sigma})\|\delta X\|_{\tilde{\alpha},\mathbb{R}^{d}}+\| \bar{R}^{y}\|_{\tilde{\alpha}, \theta-2\tilde{\alpha}+\sigma}\label{eq503}\\
		&\lesssim(\|y,y',\bar{y}'\|_{\bar{\mathbf{X}},2\tilde{\alpha},\theta}T^{\alpha-\tilde{\alpha}})\|\delta X\|_{\alpha,\mathbb{R}^{d}}+(\|\bar{R}^{y} \|_{\tilde{\alpha},\theta-\tilde{\alpha}}+ \|\bar{R}^{y} \|_{2\tilde{\alpha},\theta-2\tilde{\alpha}})T^{\tilde{\alpha}-\sigma}\notag\\
		&\lesssim  (1+\rho_{\tilde{\alpha}, I}(\bar{\mathbf{X}}))  \|y,y',\bar{y}'\|_{\bar{\mathbf{X}},2\tilde{\alpha},\theta}T^{\lambda_{0}}.\notag
	\end{align}
	
	The final task is to estimate 
	$R^{\zeta}_{t,s}:=\delta \zeta_{t,s}-\zeta'_{s}\cdot \delta X_{t,s}$.
	we rewrite 
	\begin{align*}
		R^{\zeta}_{t,s}&= 	\int^{t}_{0}S_{t,u} y_{u}\cdot\textup{d}\bar{\mathbf{X}}_{u}-	\int^{s}_{0}S_{s,u} y_{u}\cdot\textup{d}\bar{\mathbf{X}}_{u}-y_{s}\cdot \delta X_{t,s}\\
		&=(S_{t,s}-\textup{id})y_{s} \cdot\delta X_{t,s}+ (S_{t,s}-\textup{id}) \int^{s}_{0}S_{s,u} y_{u}\cdot\textup{d}\bar{\mathbf{X}}_{u}+S_{t,s}y'_{s}:\mathbb{X}_{t,s}+S_{t,s}\bar{y}'_{s}:\mathbb{X}(-r)_{t,s}+ \eta_{t,s}\\
		&:=\sum^{5}_{j=1} \Lambda^{j}_{t,s}.
	\end{align*} 
	Due to the property (\ref{eq050}), for $i=1,2$, we derive 
	\begin{align*}
		\|\Lambda^{1}_{t,s}\|_{\theta-i\tilde{\alpha}+\sigma}\lesssim \|y\|_{\infty,\theta}(t-s)^{i\tilde{\alpha}-\sigma}\rho_{\alpha, I}(\bar{\mathbf{X}})(t-s)^{\alpha},
	\end{align*}
	which implies 
	\begin{align*}
		\|\Lambda^{1}\|_{i\tilde{\alpha},\theta-i\tilde{\alpha}+\sigma}\lesssim \rho_{\alpha, I}(\bar{\mathbf{X}})T^{\alpha-\sigma}\|y,y',\bar{y}'\|_{\bar{\mathbf{X}},2\tilde{\alpha},\theta}.
	\end{align*} 
	According to (\ref{eq105}), we have
	\begin{align*}
		\|\Lambda^{2}_{t,s}\|_{\theta-i\tilde{\alpha}+\sigma}&\lesssim (t-s)^{i\tilde{\alpha}} \|\zeta\|_{\infty,\theta+\sigma}\\
		&\lesssim (t-s)^{i\tilde{\alpha}} \rho_{\alpha, I}(\bar{\mathbf{X}})  \|y,y',\bar{y}'\|_{\bar{\mathbf{X}},2\tilde{\alpha},\theta}T^{\tilde{\alpha}-\sigma},
	\end{align*}
	from which we obtain
	\begin{align*}
		\|\Lambda^{2}\|_{i\tilde{\alpha},\theta-i\tilde{\alpha}+\sigma} \lesssim  \rho_{\alpha, I}(\bar{\mathbf{X}})T^{\tilde{\alpha}-\sigma}\|y,y',\bar{y}'\|_{\bar{\mathbf{X}},2\tilde{\alpha},\theta}.
	\end{align*}
	
	Using condition (\ref{eq050}) again, we  derive
	\begin{align}
		\|\Lambda^{3}\|_{\tilde{\alpha},\theta-\tilde{\alpha}+\sigma}&\lesssim  \rho_{\alpha, I}(\bar{\mathbf{X}})\|y'\|_{\infty, \theta-\tilde{\alpha}}T^{2\alpha-\tilde{\alpha}-\sigma},\notag \\
		\|\Lambda^{3}\|_{2\tilde{\alpha},\theta-2\tilde{\alpha}+\sigma}&\lesssim  \rho_{\alpha, I}(\bar{\mathbf{X}})\|y'\|_{\infty, \theta-\tilde{\alpha}}T^{2\alpha-2\tilde{\alpha}}, \label{eq505}
	\end{align}
	and 
	\begin{align}
		\|\Lambda^{4}\|_{\tilde{\alpha},\theta-\tilde{\alpha}+\sigma}&\lesssim  \rho_{\alpha, I}(\bar{\mathbf{X}})\|\bar{y}'\|_{\infty, \theta-\tilde{\alpha}}T^{2\alpha-\tilde{\alpha}-\sigma},\notag \\
		\|\Lambda^{4}\|_{2\tilde{\alpha},\theta-2\tilde{\alpha}+\sigma}&\lesssim  \rho_{\alpha, I}(\bar{\mathbf{X}})\|\bar{y}'\|_{\infty, \theta-\tilde{\alpha}}T^{2\alpha-2\tilde{\alpha}}. \label{eq506}
	\end{align}
	And, applying Lemma \ref{lem003} again, for $i=1,2$, we obtain 
	\begin{align*}
		\|\Lambda^{5}\|_{i\tilde{\alpha},\theta-i\tilde{\alpha}+\sigma}\lesssim \rho_{\alpha, I}(\bar{\mathbf{X}})T^{\tilde{\alpha}-\sigma}\|y,y',\bar{y}'\|_{\bar{\mathbf{X}},2\tilde{\alpha},\theta}.
	\end{align*}
	
	The result $(iii)$  can be obtained by modifying the proof of $(ii)$. In fact,  we can use 
	\begin{align*}
		\|y'\|_{\infty,\theta-2\tilde{\alpha}+\sigma}&\lesssim 
		\sup_{t\in[0,T]}\frac{\|y'_{t}-y'_{0}\|_{\theta-2\tilde{\alpha}+\sigma}}{t^{\tilde{\alpha}-\sigma}} t^{\tilde{\alpha}-\sigma}+ \|y'_{0}\|_{\theta-2\tilde{\alpha}+\sigma}\\
		&\lesssim \|y'\|_{\infty, \theta-\tilde{\alpha}}^{\frac{\sigma} {\tilde{\alpha}}}\|\delta y'\|_{\tilde{\alpha}-\sigma, \theta-2\tilde{\alpha}}^{\frac{\tilde{\alpha}-\sigma} {\tilde{\alpha}}} T^{\tilde{\alpha}-\sigma}+ \|y'_{0}\|_{\theta-\tilde{\alpha}},
	\end{align*}
	and 
	\begin{align*}
		\|\bar{y}'\|_{\infty,\theta-2\tilde{\alpha}+\sigma}&\lesssim 
		\sup_{t\in[0,T]}\frac{\|\bar{y}'_{t}-\bar{y}'_{0}\|_{\theta-2\tilde{\alpha}+\sigma}}{t^{\tilde{\alpha}-\sigma}} t^{\tilde{\alpha}-\sigma}+ \|\bar{y}'_{0}\|_{\theta-2\tilde{\alpha}+\sigma}\\
		&\lesssim \|\bar{y}'\|_{\infty, \theta-\tilde{\alpha}}^{\frac{\sigma} {\tilde{\alpha}}}\|\delta \bar{y}'\|_{\tilde{\alpha}-\sigma, \theta-2\tilde{\alpha}}^{\frac{\tilde{\alpha}-\sigma} {\tilde{\alpha}}} T^{\tilde{\alpha}-\sigma}+ \|\bar{y}'_{0}\|_{\theta-\tilde{\alpha}}.
	\end{align*}
	to estimate (\ref{eq503}), (\ref{eq505})  and (\ref{eq506}). Then, it follows the result as desired.
\end{Proof}

	\begin{Proof} Lemma \ref{lem007}.

	By Lemma \ref{lem006}, we know $\chi'_{t}=z_{t}$. Then, interpolation inequality yields 
	\begin{align*}
		\|\zeta'-\chi'\|_{\infty, \theta+\sigma-\alpha}&=	\|y-z\|_{\infty, \theta+\sigma-\alpha}\\
		& \lesssim  \rho_{2\tilde{\alpha},2\alpha, \theta}(y,z)T^{\lambda}+ \|y(0)-z(0)\|_{\theta}.
	\end{align*}

	And, it is not hard to deduce 
	\begin{align*}
		&~~~~\|\delta y_{t,s}-\delta z_{t,s}\|_{\theta+\sigma-2\alpha}\\
		&=\|y'_{s}\cdot\delta X_{t,s}+\bar{y}'_{s}\cdot\delta X_{t-r,s-r}+\bar{R}^{y}_{t,s}-z'_{s}\cdot\delta Y_{t,s}-\bar{z}'_{s}\cdot\delta y_{t-r,s-r}-\bar{R}^{z}_{t,s}\|_{\theta+\sigma-2\alpha}\\
		&\lesssim \rho_{\alpha, I}(\bar{\mathbf{X}})[\|y'-z'\|_{\infty,\theta-\alpha}+\|\bar{y}'-\bar{z}'\|_{\infty,\theta-\alpha}] (t-s)^{\alpha}+ \|\bar{R}^{y}_{t,s}-\bar{R}^{z}_{t,s}\|_{\theta+\sigma-2\alpha}\\
		&~~~+\rho_{\alpha, I}(\bar{\mathbf{X}}, \bar{\mathbf{Y}})(\|z'\|_{\infty,\theta-\alpha}+\|\bar{z}'\|_{\infty,\theta-\alpha})(t-s)^{\alpha}.
	\end{align*}
	And, we note 
	\begin{align*}
		\|\bar{R}^{y}-\bar{R}^{z}\|_{\tilde{\alpha},\theta+\sigma-2\alpha}\lesssim (\|\bar{R}^{y}-\bar{R}^{z}\|_{\tilde{\alpha},\theta-\alpha}+ \|\bar{R}^{y}-\bar{R}^{z}\|_{2\tilde{\alpha},\theta-2\alpha})T^{\frac{\tilde{\alpha}(\alpha-\sigma)}{\alpha}}
	\end{align*}
	Thus, it holds
	\begin{align*}
		\|\delta\zeta'-\delta\chi'\|_{\tilde{\alpha},\theta+\sigma-\alpha}=\|\delta y-\delta z\|_{\tilde{\alpha},\theta+\sigma-\alpha}\lesssim 	\rho_{2\tilde{\alpha},2\alpha, \theta}(y,z)T^{\lambda}+\rho_{\alpha, I}(\bar{\mathbf{X}}, \bar{\mathbf{Y}}).
	\end{align*}

	Set 
	\begin{align*}
		\hat{f}_{t,s}=y_{s}\cdot\delta X_{t,s}+y'_{s}: \mathbb{X}_{t,s}+\bar{y}'_{s}: \mathbb{X}(-r)_{t,s}-
		z_{s}\cdot\delta Y_{t,s}-z'_{s}: \mathbb{Y}_{t,s}-\bar{z}'_{s}: \mathbb{Y}(-r)_{t,s},
	\end{align*}
	and
	\begin{align*}
		\hat{\eta}_{t,s}=\int^{t}_{s}S_{t,u}y_{u}\cdot \textup{d}\bar{\mathbf{X}}_{u}- \int^{t}_{s}S_{t,u}z_{u}\cdot \textup{d}\bar{\mathbf{Y}}-S_{t,s}\hat{f}_{t,s}.
	\end{align*}
	Employing Lemma \ref{lem003}, we obtain 
	\begin{align*}
		\|\hat{\eta}_{t,0}\|_{\theta+\sigma}\lesssim \|\hat{f}\|_{\mathcal{J}^{\tilde{\alpha}}_{\theta-2\alpha+2\tilde{\alpha}}}t^{3\tilde{\alpha}-2\alpha-\sigma}
		\lesssim  \rho_{2\tilde{\alpha},2\alpha, \theta}(y,z)t^{3\tilde{\alpha}-2\alpha-\sigma}+\rho_{\alpha,I}(\bar{\mathbf{X}},\bar{\mathbf{Y}}).
	\end{align*}
	And, it is easy to derive 
	\begin{align*}
		\|S_{t,0}\hat{f}_{t,0}\|_{\theta+\sigma}\lesssim  \rho_{2\tilde{\alpha},2\alpha, \theta}(y,z)t^{\alpha-\sigma}+\rho_{\alpha,I}(\bar{\mathbf{X}},\bar{\mathbf{Y}}).
	\end{align*}
	Thus, it holds 
	\begin{align}
		\|\zeta-\chi\|_{\infty,\theta+\sigma}\lesssim \rho_{2\tilde{\alpha},2\alpha, \theta}(y,z)T^{\lambda}+\rho_{\alpha,I}(\bar{\mathbf{X}},\bar{\mathbf{Y}}).
	\end{align}
	
	Note
	\begin{align*}
		R^{\zeta}_{t,s}-R^{\chi}_{t,s}:&=\delta\zeta_{t,s}-\delta\chi_{t,s}-\zeta'_{s}\cdot \delta X_{t,s}+\chi'_{s}\cdot \delta Y_{t,s}.
	\end{align*}
	Similar to the proof in Lemma \ref{lem006}, for $i=1,2$, it follows 
	\begin{align*}
		\|R^{\zeta}-R^{\chi}\|_{i\tilde{\alpha},\theta+\sigma-i\alpha}\lesssim \rho_{2\tilde{\alpha},2\alpha, \theta}(y,z)T^{\lambda}+\rho_{\alpha,I}(\bar{\mathbf{X}},\bar{\mathbf{Y}}).
	\end{align*}
	Consequently, the proof is completed.
\end{Proof}

\section*{Acknowledgements}
The work is supported in part by the NSFC Grant Nos. 12171084 and the
Fundamental Research Funds for the Central Universities No. RF1028623037.
\section*{Data Availability Statements}
Data sharing not applicable to this article as no datasets were generated or analyzed during the current study.
\section*{Conflict of Interest}
The authors declare that they have no conflict of interest.
	

\begin{thebibliography}{99}
		\normalsize
		\parskip=0.3ex
		\bibitem{MR4100851}
		C.~Bayer, D.~Belomestny, M.~Redmann, S.~Riedel, and J.~Schoenmakers.
		\newblock Solving linear parabolic rough partial differential equations.
		\newblock {\em J. Math. Anal. Appl.}, 490(1):124236, 45, 2020.
		
		\bibitem{MR4117091}
		M.~Besal\'{u}, G.~Binotto, and C.~Rovira.
		\newblock Convergence of delay equations driven by a {H}\"{o}lder continuous
		function of order {$1/3<\beta<1/2$}.
		\newblock {\em Electron. J. Differential Equations}, pages Paper No. 65, 27,
		2020.
		
		\bibitem{MR3225810}
		M.~Besal\'{u}, D.~M\'{a}rquez-Carreras, and C.~Rovira.
		\newblock Delay equations with non-negativity constraints driven by a
		{H}\"{o}lder continuous function of order {$\beta\in(\frac13,\frac12)$}.
		\newblock {\em Potential Anal.}, 41(1):117--141, 2014.
		
		\bibitem{MR2836524}
		M.~Besal\'{u} and D.~Nualart.
		\newblock Estimates for the solution to stochastic differential equations
		driven by a fractional {B}rownian motion with {H}urst parameter
		{$H\in(\frac13,\frac12)$}.
		\newblock {\em Stoch. Dyn.}, 11(2-3):243--263, 2011.
		
		\bibitem{MR2888697}
		M.~Besal\'{u} and C.~Rovira.
		\newblock Stochastic delay equations with non-negativity constraints driven by
		fractional {B}rownian motion.
		\newblock {\em Bernoulli}, 18(1):24--45, 2012.
		
		\bibitem{MR3688537}
		B.~Boufoussi and S.~Hajji.
		\newblock Stochastic delay differential equations in a {H}ilbert space driven
		by fractional {B}rownian motion.
		\newblock {\em Statist. Probab. Lett.}, 129:222--229, 2017.
		
		\bibitem{MR2343865}
		I.~Boutle, R.~H.~S. Taylor, and R.~A. R\"{o}mer.
		\newblock El {N}i\~{n}o and the delayed action oscillator.
		\newblock {\em Amer. J. Phys.}, 75(1):15--24, 2007.
		
		\bibitem{MR4551604}
		Q.~Cao, H.~Gao, and B.~Schmalfuss.
		\newblock Wong-{Z}akai type approximations of rough random dynamical systems by
		smooth noise.
		\newblock {\em J. Differential Equations}, 358:218--255, 2023.
		
		\bibitem{MR2803093}
		T.~Caraballo, M.~J. Garrido-Atienza, and T.~Taniguchi.
		\newblock The existence and exponential behavior of solutions to stochastic
		delay evolution equations with a fractional {B}rownian motion.
		\newblock {\em Nonlinear Anal.}, 74(11):3671--3684, 2011.
		
		\bibitem{MR2891223}
		A.~Deya.
		\newblock Numerical schemes for rough parabolic equations.
		\newblock {\em Appl. Math. Optim.}, 65(2):253--292, 2012.
		
		\bibitem{MR3957994}
		A.~Deya, M.~Gubinelli, M.~Hofmanov\'{a}, and S.~Tindel.
		\newblock A priori estimates for rough {PDE}s with application to rough
		conservation laws.
		\newblock {\em J. Funct. Anal.}, 276(12):3577--3645, 2019.
		
		\bibitem{MR3146369}
		M.~A. Diop and M.~J. Garrido-Atienza.
		\newblock Retarded evolution systems driven by fractional {B}rownian motion
		with {H}urst parameter {$H>1/2$}.
		\newblock {\em Nonlinear Anal.}, 97:15--29, 2014.
		
		\bibitem{MR4385780}
		L.~H. Duc.
		\newblock Random attractors for dissipative systems with rough noises.
		\newblock {\em Discrete Contin. Dyn. Syst.}, 42(4):1873--1902, 2022.
		
		\bibitem{MR4656793}
		L.~H. Duc and P.~Kloeden.
		\newblock Numerical attractors for rough differential equations.
		\newblock {\em SIAM J. Numer. Anal.}, 61(5):2381--2407, 2023.
		
		\bibitem{MR2202322}
		M.~Ferrante and C.~Rovira.
		\newblock Stochastic delay differential equations driven by fractional
		{B}rownian motion with {H}urst parameter {$H>{1\over2}$}.
		\newblock {\em Bernoulli}, 12(1):85--100, 2006.
		
		\bibitem{MR2737158}
		M.~Ferrante and C.~Rovira.
		\newblock Convergence of delay differential equations driven by fractional
		{B}rownian motion.
		\newblock {\em J. Evol. Equ.}, 10(4):761--783, 2010.
		
		\bibitem{MR4174393}
		P.~K. Friz and M.~Hairer.
		\newblock {\em A course on rough paths}.
		\newblock Universitext. Springer, Cham, second edition, [2020] \copyright 2020.
		\newblock With an introduction to regularity structures.
		
		\bibitem{MR4266219}
		H.~Gao, M.~J. Garrido-Atienza, A.~Gu, K.~Lu, and B.~Schmalfu\ss.
		\newblock Rough path theory to approximate random dynamical systems.
		\newblock {\em SIAM J. Appl. Dyn. Syst.}, 20(2):997--1021, 2021.
		
		\bibitem{MR4040992}
		A.~Gerasimovi\v{c}s and M.~Hairer.
		\newblock H\"{o}rmander's theorem for semilinear {SPDE}s.
		\newblock {\em Electron. J. Probab.}, 24:Paper No. 132, 56, 2019.
		
		\bibitem{MR4299812}
		A.~Gerasimovi\v{c}s, A.~Hocquet, and T.~Nilssen.
		\newblock Non-autonomous rough semilinear {PDE}s and the multiplicative sewing
		lemma.
		\newblock {\em J. Funct. Anal.}, 281(10):Paper No. 109200, 65, 2021.
		
		\bibitem{MR4251893}
		M.~Ghani~Varzaneh, S.~Riedel, and M.~Scheutzow.
		\newblock A dynamical theory for singular stochastic delay differential
		equations {II}: {N}onlinear equations and invariant manifolds.
		\newblock {\em Discrete Contin. Dyn. Syst. Ser. B}, 26(8):4587--4612, 2021.
		
		\bibitem{MR4385646}
		M.~Ghani~Varzaneh, S.~Riedel, and M.~Scheutzow.
		\newblock A dynamical theory for singular stochastic delay differential
		equations {I}: linear equations and a multiplicative ergodic theorem on
		fields of {B}anach spaces.
		\newblock {\em SIAM J. Appl. Dyn. Syst.}, 21(1):542--587, 2022.
		
		\bibitem{MR2523298}
		S.~A. Gourley and Y.~Kuang.
		\newblock Dynamics of a neutral delay equation for an insect population with
		long larval and short adult phases.
		\newblock {\em J. Differential Equations}, 246(12):4653--4669, 2009.
		
		\bibitem{MR2599193}
		M.~Gubinelli and S.~Tindel.
		\newblock Rough evolution equations.
		\newblock {\em Ann. Probab.}, 38(1):1--75, 2010.
		
		\bibitem{MR4097587}
		R.~Hesse and A.~Neam\c{t}u.
		\newblock Global solutions and random dynamical systems for rough evolution
		equations.
		\newblock {\em Discrete Contin. Dyn. Syst. Ser. B}, 25(7):2723--2748, 2020.
		
		\bibitem{MR4431448}
		R.~Hesse and A.~Neam\c{t}u.
		\newblock Global solutions for semilinear rough partial differential equations.
		\newblock {\em Stoch. Dyn.}, 22(2):Paper No. 2240011, 18, 2022.
		
		\bibitem{MR3797622}
		A.~Hocquet and M.~Hofmanov\'{a}.
		\newblock An energy method for rough partial differential equations.
		\newblock {\em J. Differential Equations}, 265(4):1407--1466, 2018.
		
		\bibitem{MR4074697}
		A.~Hocquet, T.~Nilssen, and W.~Stannat.
		\newblock Generalized {B}urgers equation with rough transport noise.
		\newblock {\em Stochastic Process. Appl.}, 130(4):2159--2184, 2020.
		
		\bibitem{MR2471936}
		Y.~Hu and D.~Nualart.
		\newblock Rough path analysis via fractional calculus.
		\newblock {\em Trans. Amer. Math. Soc.}, 361(5):2689--2718, 2009.
		
		\bibitem{MR4567460}
		C.~Kuehn and A.~Neam\c{t}u.
		\newblock Center manifolds for rough partial differential equations.
		\newblock {\em Electron. J. Probab.}, 28:Paper No. 48, 31, 2023.
		
		\bibitem{MR1654527}
		T.~J. Lyons.
		\newblock Differential equations driven by rough signals.
		\newblock {\em Rev. Mat. Iberoamericana}, 14(2):215--310, 1998.
		
		\bibitem{MR2036784}
		T.~J. Lyons and Z.~Qian.
		\newblock {\em System control and rough paths}.
		\newblock Oxford Mathematical Monographs. Oxford University Press, Oxford,
		2002.
		\newblock Oxford Science Publications.
		
		\bibitem{MR4560602}
		H.~Ma and H.~Gao.
		\newblock Unstable manifolds for rough evolution equations.
		\newblock {\em Stoch. Dyn.}, 22(8):Paper No. 2240033, 33, 2022.
		
		\bibitem{MR4284415}
		A.~Neam\c{t}u and C.~Kuehn.
		\newblock Rough center manifolds.
		\newblock {\em SIAM J. Math. Anal.}, 53(4):3912--3957, 2021.
		
		\bibitem{MR2453555}
		A.~Neuenkirch, I.~Nourdin, and S.~Tindel.
		\newblock Delay equations driven by rough paths.
		\newblock {\em Electron. J. Probab.}, 13:no. 67, 2031--2068, 2008.
		
		\bibitem{MR1386683}
		W.~M. Ruess.
		\newblock Existence of solutions to partial functional-differential equations
		with delay.
		\newblock In {\em Theory and applications of nonlinear operators of accretive
			and monotone type}, volume 178 of {\em Lecture Notes in Pure and Appl.
			Math.}, pages 259--288. Dekker, New York, 1996.
		
		\bibitem{MR2120285}
		G.~Stoica.
		\newblock A stochastic delay financial model.
		\newblock {\em Proc. Amer. Math. Soc.}, 133(6):1837--1841, 2005.
		
		\bibitem{MR3236092}
		S.~Tindel and I.~Torrecilla.
		\newblock Some differential systems driven by a f{B}m with {H}urst parameter
		greater than 1/4.
		\newblock In {\em Stochastic analysis and related topics}, volume~22 of {\em
			Springer Proc. Math. Stat.}, pages 169--202. Springer, Heidelberg, 2012.
		
		\bibitem{MR1415838}
		J.~Wu.
		\newblock {\em Theory and applications of partial functional-differential
			equations}, volume 119 of {\em Applied Mathematical Sciences}.
		\newblock Springer-Verlag, New York, 1996.
		
		\bibitem{MR4608383}
		Q.~Yang, X.~Lin, and C.~Zeng.
		\newblock Random attractors for rough stochastic partial differential
		equations.
		\newblock {\em J. Differential Equations}, 371:50--82, 2023.
		
		
		
\end{thebibliography}
\end{document}